\documentclass[11pt,a4paper]{article}
\usepackage[utf8]{inputenc}     % utf8 support 
\usepackage[T1]{fontenc}        % code for pdf file 

\usepackage[left=3cm, top=2.5cm,bottom=2.5cm,right=3cm]{geometry}

\usepackage[nointlimits]{amsmath}
\usepackage{mathtools}
\usepackage{amssymb}
\usepackage{amsthm}
\usepackage{amsfonts}
\usepackage{mathrsfs}
\usepackage{calc}
\usepackage{graphicx}
\usepackage{xcolor}
\usepackage{dsfont}
\usepackage{tikz}
\usepackage{pgfplots}
\pgfplotsset{compat=1.14} % for backward compatibility
\usepackage{bm}
\usepackage[british]{babel}
\usepackage{csquotes}
\usepackage{hyperref}
\hypersetup{hidelinks}

\usepackage[
backend=biber,
style=numeric,
sorting=nyt,
sortcites=true,
doi=true,
isbn=false,
url=false,
maxbibnames=99,
]{biblatex}
\usepackage{comment}

\numberwithin{equation}{section}
\newtheorem {theorem}{Theorem}[section]
\newtheorem*{maintheorem}{Main Theorem}

\newtheorem {lemma}[theorem]{Lemma}
\newtheorem {corollary}[theorem]{Corollary}

{\theoremstyle{definition}

\newtheorem {remark}[theorem]{Remark}
}
{\theoremstyle{plain}

}

\newcommand{\mres}{\mathbin{\vrule height 1.6ex depth 0pt width
0.13ex\vrule height 0.13ex depth 0pt width 1.3ex}}
\newcommand{\dint}{\textup{d}}
\newcommand{\ind}[1]{ \mathds{1} \left\{ #1 \right\} }
\DeclarePairedDelimiterX{\norm}[1]{\lVert}{\rVert}{#1}
\newcommand{\sphere}{S}
\newcommand{\ball}{B}

\def\EE{\mathbb{E}}

\def\NN{\mathbb{N}}

\def\PP{\mathbb{P}}

\def\RR{\mathbb{R}}

\def\sfN{{\sf N}}

\def\cA{\mathcal{A}}

\def\cC{\mathcal{C}}

\def\cF{\mathcal{F}}

\usepackage{authblk}
\newcommand\blfootnote[1]{%
  \begingroup
  \renewcommand\thefootnote{}\footnote{#1}%
  \addtocounter{footnote}{-1}%
  \endgroup
}
\title{\bf Weak convergence of the intersection point process of Poisson hyperplanes}
\author[1]{Anastas Baci}
\author[2]{Gilles Bonnet}
\author[3]{Christoph Thäle}
\affil[1,2,3]{\small{\textit{Faculty of Mathematics, Ruhr University Bochum, Germany}}}
\date{}

\begin{document}
%%%%%%%%%%%%%%%%%%%%%%%%%%%%%%%%%%%%%%%
%%%%%%%%%%%%%%%%%%%%%%%%%%%%%%%%%%%%%%%

\maketitle

\begin{abstract}
\noindent  This paper deals with the intersection point process of a stationary and isotropic Poisson hyperplane process in $\mathbb{R}^d$ of intensity $t>0$, where only hyperplanes that intersect a centred ball of radius $R>0$ are considered. Taking $R=t^{-\frac{d}{d+1}}$ it is shown that this point process converges in distribution, as $t\to\infty$, to a Poisson point process on $\mathbb{R}^d\setminus\{0\}$ whose intensity measure has power-law density proportional to $\|x\|^{-(d+1)}$ with respect to the Lebesgue measure. A bound on the speed of convergence in terms of the Kantorovich-Rubinstein distance is provided as well. In the background is a general functional Poisson approximation theorem on abstract Poisson spaces. Implications on the weak convergence of the convex hull of the intersection point process and the convergence of its $f$-vector are also discussed, disproving and correcting thereby a conjecture of Devroye and Toussaint [J.\ Algorithms 14.3 (1993), 381--394] in computational geometry.
\end{abstract}

%\tableofcontents

\blfootnote{\textsuperscript{1}E-mail address: \texttt {anastas.baci@rub.de}}
\blfootnote{\textsuperscript{2}E-mail address: \texttt{gilles.bonnet@rub.de}}
\blfootnote{\textsuperscript{3}E-mail address: \texttt{christoph.thaele@rub.de}}

%%%%%%%%%%%%%%%%%%%%%%%%%%%%%%%%%%%%%%%
\section{Introduction}
%%%%%%%%%%%%%%%%%%%%%%%%%%%%%%%%%%%%%%%

The mathematical analysis of Poisson point processes of hyperplanes in $\RR^d$ ($d\geq 2$) and the resulting random tessellations has a long tradition in stochastic geometry, see, for example, \cite{matheron1975random,schneider2008stochastic,stoyan1995stochastic} and the many references given therein. A large number of mean value formulas and relations are known explicitly and in the last decade also a comprehensive second-order and central limit theory for a variety of geometric functionals associated with such Poisson hyperplane tessellations has been developed, see \cite{last2014moments,HeinrichCLT09,heinrich2007limit,HeinrichSchmidtzSchmidt,ReitznerSchulte} and also Remark \ref{remZeroTyp} below.
 
The problem we are dealing with in this paper is motivated by an observation known in the area of computational geometry: \textit{An arrangement of $n$ lines chosen at random from $\RR^2$ has a vertex set whose convex hull has an absolutely bounded expected vertex number}, see \cite{berend2005convex,devroye1993convex,golin2003convex}.
We refer the reader to the introduction of \cite{berend2005convex} for an explanation on how this fact improves the average algorithmic complexity for the computation of the convex hull of such vertex sets.
In \cite{devroye1993convex} it is conjectured that the expected vertex number tends to $4$, as the number of random lines goes to infinity. Our results can be considered as a variation and extension of this theme in a number of different directions. To describe them, we start by considering for an arbitrary space dimension $d\geq 2$ an arrangement of random hyperplanes in $\RR^d$ driven by a stationary and isotropic Poisson point process on the space of hyperplanes having intensity $t>0$. This gives rise to an infinite collection of random hyperplanes and the convex hull of the associated random set of intersection points almost surely coincides with $\RR^d$. Therefore, in a next step, we apply a thinning procedure to the Poisson process and we only keep those hyperplanes which intersect a ball of radius $R>0$ centred at the origin. This leaves us with an almost surely finite collection of random hyperplanes in $\RR^d$ in general position. In particular, any $d$-tuple of the remaining distinct hyperplanes intersects in a randomly located point in $\RR^d$. By $\Xi_{t,R}$ we denote the random point process of all such intersection points, see Figure~\ref{fig:simulation} % and Figure~\ref{fig:simulation2} 
for illustrations and simulations. The convex hull of $\Xi_{t,R}$ is a random polytope in $\RR^d$ whose description is in the focus of our attention and whose analysis is motivated by \cite{berend2005convex,devroye1993convex,golin2003convex,atallah1986computing,ching1985finding}. Our main result is a quantitative limit theorem for the point process $\Xi_{t,R}$ itself, as $t\to\infty$ and where $R=R(t)$ is a suitable function of $t$. As it turns out, if we choose $R=t^{-\frac{d}{d+1}}$ then $\Xi_{t,R}$ converges in distribution to a non-trivial limiting point process, namely a Poisson point process $\zeta$ in $\RR^d\setminus\{0\}$ whose intensity measure has the power-law density $x\mapsto C_d\|x\|^{-(d+1)}$ with respect to the Lebesgue measure and where $C_d$ is a constant only depending on the dimension $d$ (for example, $C_2=\frac{4}{3\pi^2}$). Formally, our main result reads as follows.

\begin{maintheorem}[Theorem \ref{Poissonapproximation}]\label{thm:intro}
    Let $\Xi_{t,R}$ be the intersection point process and $\zeta$ be the Poisson point process as introduced above. Then, taking $R=t^{-\frac{d}{d+1}}$, we have that $\Xi_{t,R}$ converges to $\zeta$ in distribution (on the space of point process in $\RR^d\setminus\{0\}$ supplied with the vague topology), as $t\to\infty$.
\end{maintheorem}

As a particular feature, we will derive an upper bound on the speed of convergence in this limit theorem, which is measured in terms of the famous Kantorovich-Rubinstein optimal transportation distance. Moreover, this convergence together with a continuous-mapping-type result implies that the convex hull of $\Xi_{t,R}$ converges in distribution to the convex hull of $\zeta$ on the space of convex bodies in $\RR^d$ supplied with the Hausdorff distance. We will also conclude from this the convergence in distribution of the number of $k$-dimensional faces for any $k\in\{0,1,\ldots,d-1\}$. We argue that in dimension $d=2$, for example, the expected number of vertices cannot converge to a constant smaller than $\pi^2/2\approx 4.93$, disproving and correcting thereby the aforementioned conjecture from \cite{devroye1993convex}. 

We would like to emphasize that Poisson point processes in $\RR^d$ with a power-law density function, such as the one appearing in the definition of $\zeta$, and their convex hulls have recently been intensively investigated and are closely connected with (asymptotic) geometric and probabilistic description of several well known objects in stochastic geometry. As concrete examples we mention the Poisson zero polytope, the typical Poisson-Voronoi cell, random convex hulls of independent random points in convex bodies with smooth boundary, spherical random polytopes on half-spheres or Poisson-Voronoi tessellations in spherical spaces. Our paper adds another item to this list and underlines once again the outstanding role in stochastic geometry of Poisson point processes with a power-law density function and their convex hulls, see \cite{kabluchko2019cones,kabluchko2018beta,kabluchko2019angles,kabluchko2019expected,kabluchko2019typical}.

In order to prove the results outlined above, we have to bring together a number of different technical tools and devices. To show the convergence of the intersection point process $\Xi_{t,R}$ to the Poisson point process $\zeta$ we make use of an abstract functional Poisson limit theorem which has been developed in \cite{decreusefond2016functional} in the framework of the Malliavin-Stein technique on abstract Poisson spaces. It is this device, which delivers the rate of convergence measured in the Kantorovich-Rubinstein for point processes. Beside the convergence of the intensity measures of the involved point processes, this approach also requires control over a number of auxiliary integral expressions. Their analysis constitutes the most technical part of this paper and requires an extensive use of integral-geometric transformation formulas of Blaschke-Petkantschin type in combination with asymptotic expansions of integrals having a geometric flavour. This also results in a higher-dimensional asymptotic version of Huygen's principle

\medspace

\begin{figure}
    \centering
    \begin{tikzpicture}
        \def\mybottom{-2.87};
        \def\mytop{3.42};
        \def\myleft{-2.92};
        \def\myright{3.38};
        \def\mycenterx{0.224}
        \def\mycentery{0.268}
        \def\myhalfside{0.36}
        \def\mywidth{0.467\textwidth}
        \def\myinnerbottom{\mycentery-\myhalfside};
        \def\myinnertop{\mycentery+\myhalfside};
        \def\myinnerleft{\mycenterx-\myhalfside};
        \def\myinnerright{\mycenterx+\myhalfside};
        \begin{scope}[shift={(0,0)}]
            \node[inner sep=0pt] (thousand) at (0,0)
            {\includegraphics[trim= 135 60 125 75, clip, width=\mywidth]{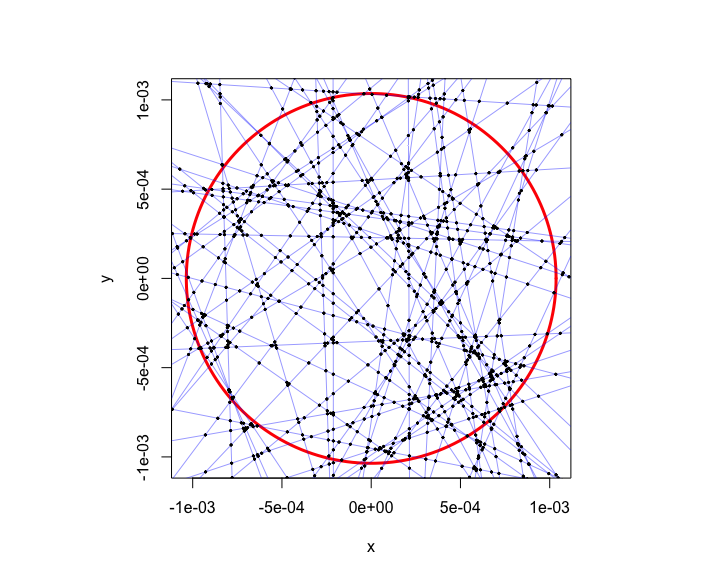}};
            \node[inner sep=0pt] (toprightthausend) at (\myright,\mytop) {};
            \node[inner sep=0pt] (bottomrightthausend) at (\myright,\mybottom) {};
        \end{scope}
        \begin{scope}[shift={(8,0)}]
            \node[inner sep=0pt] (hundred) at (0,0)
            {\includegraphics[trim= 135 60 125 75, clip, width=\mywidth]{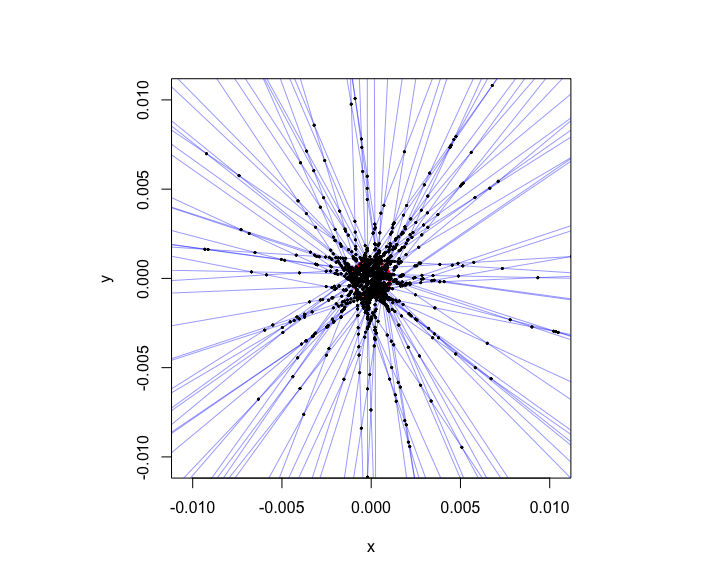}};
            \draw[red,thick] (\myinnerleft,\myinnerbottom) rectangle (\myinnerright,\myinnertop);
            \draw[thick,red] (\myinnerright,\myinnertop) -- (toprightthausend);
            \draw[thick,red] (\myinnerright,\myinnerbottom) -- (bottomrightthausend);
            \node[inner sep=0pt] (bottomleftthundred) at (\myleft,\mybottom) {};
            \node[inner sep=0pt] (bottomrighthundred) at (\myright,\mybottom) {};
        \end{scope}
        \begin{scope}[shift={(8,-8)}]
            \node[inner sep=0pt] (ten) at (0,0)
            {\includegraphics[trim= 135 60 125 75, clip, width=\mywidth]{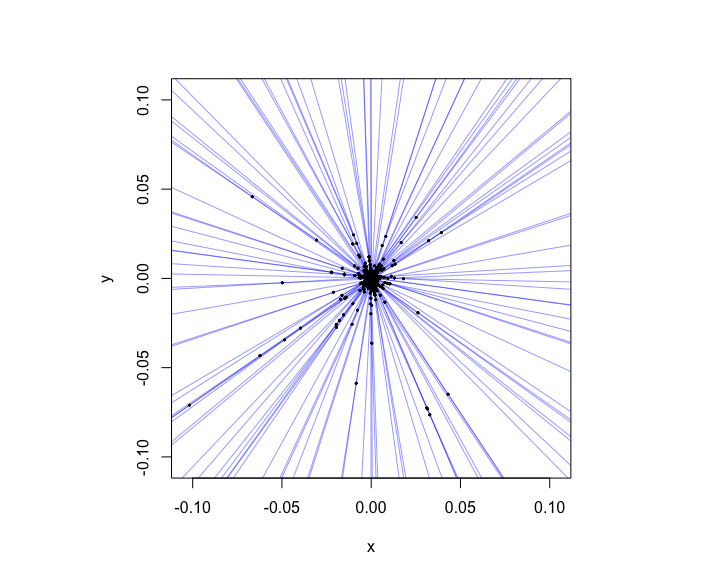}};
            \draw[red,thick] (\myinnerleft,\myinnerbottom) rectangle (\myinnerright,\myinnertop);
            \draw[thick,red] (\myinnerleft,\myinnerbottom) -- (bottomleftthundred);
            \draw[thick,red] (\myinnerright,\myinnerbottom) -- (bottomrighthundred);
            \node[inner sep=0pt] (bottomleftten) at (\myleft,\mybottom) {};
            \node[inner sep=0pt] (topleftten) at (\myleft,\mytop) {};
        \end{scope}
        \begin{scope}[shift={(0,-8)}]
            \node[inner sep=0pt] (one) at (0,0)
            {\includegraphics[trim= 135 60 125 75, clip, width=\mywidth]{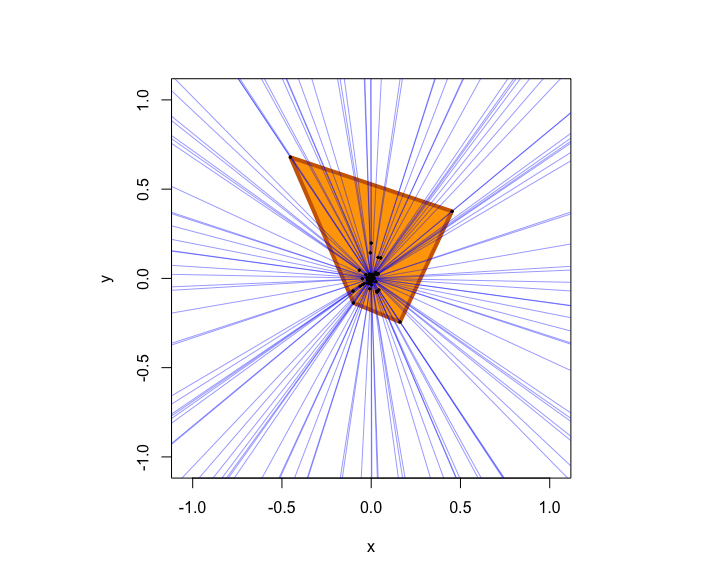}};
            \draw[red,thick] (\myinnerleft,\myinnerbottom) rectangle (\myinnerright,\myinnertop);
            \draw[thick,red] (\myinnerleft,\myinnertop) -- (topleftten);
            \draw[thick,red] (\myinnerleft,\myinnerbottom) -- (bottomleftten);
        \end{scope}
    \end{tikzpicture}
    \caption{One single realisation, visualised at four different scales, of the restricted Poisson hyperplane process in $\RR^2$ (blue lines) of intensity $t=30\,000$, the ball of radius $R=t^{-2/3}\approx 0.001$ (red circle), the intersection point process $\Xi_{t,R}$ (black points) and its convex hull (orange polygon).}
    \label{fig:simulation}
\end{figure}
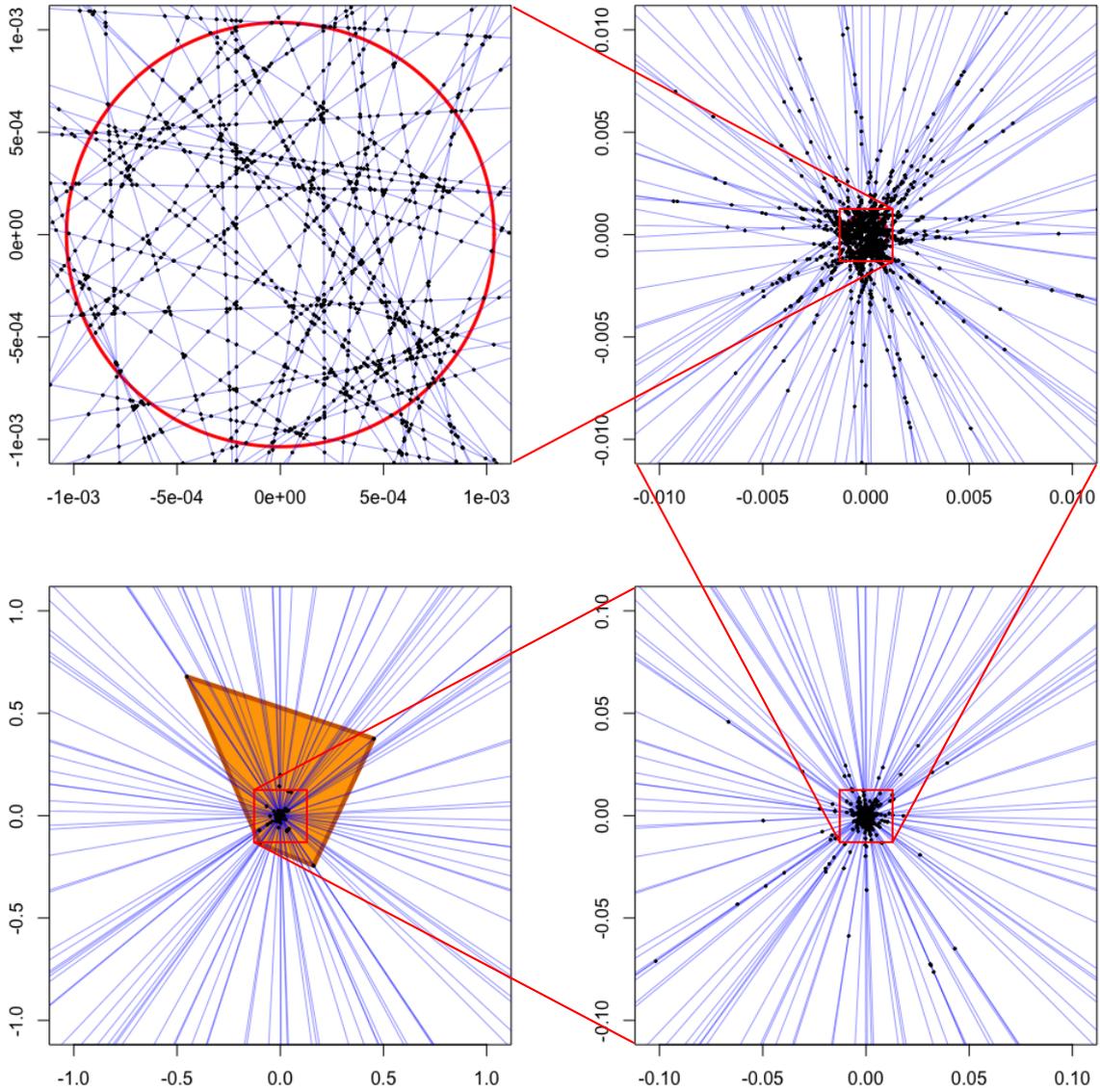

%\begin{figure}[t]
%    \centering
    %\includegraphics[width=0.6\textwidth]{interpts.png}
%    \includegraphics[width=0.49\textwidth]{IP200ZoomIn.png}
%    \hfill
%    \includegraphics[width=0.49\textwidth]{IP200ZoomOut.png}
%    \caption{Simulation of the restricted Poisson hyperplane process in $\RR^2$ (blue lines) with $t=200$, the ball of radius $R=t^{-2/3}\simeq 0.029$ (red circle) and the intersection point process $\Xi_{t,R}$ (black points). % restricted to a window.
%    Both graphics represent the same realisation of the process viewed at different scales.}
%    \label{fig:simulation}
%\end{figure}

The remaining parts of this paper are structured as follows. In Section \ref{sec:Prelim} we recall some preliminaries, mainly related to Grassmannians and integral geometry. We also formally introduce there the Poisson hyperplane process and its intersection point process, which are the main objects we deal with in this paper. Section \ref{sec:preplemmas} contains a number of preparatory lemmas, which are of more technical nature, but which are essential in the proof of our main result. The convergence of the intensity measure of the intersection point process is derived in Section \ref{sec:ConvIntensityMeasure} on the basis of the material developed in the previous section. In Section \ref{sec:ConvPPandConv} we prove our quantitative limit theorem for the intersection point process and discuss implications to the convergence of its convex hull. Also, we elaborate there on the implication of our results to the conjecture of Devroye and Toussaint from \cite{devroye1993convex}.

%%%%%%%%%%%%%%%%%%%%%%%%%%%%%%%%%%%%%%%
\section{Notation and preliminaries}\label{sec:Prelim}
%%%%%%%%%%%%%%%%%%%%%%%%%%%%%%%%%%%%%%%

\subsection{General notation, linear and affine Grassmanians}

In this article we work in the $d$-dimensional Euclidean space $\RR^d$, $d\geq 2$.
It is equipped with the usual topology and Lebesgue measure. We consider the standard scalar product $\langle\,\cdot\,,\,\cdot\,\rangle$ and norm $\norm{\,\cdot\,}$, and denote by $u^{\perp}$ and $E^\perp$ the orthogonal spaces of a vector $u\in\RR^d$ and a flat $E\subset\RR^d$.

The unit ball and sphere (centred at the origin) are denoted by $\ball^d:=\{x\in\RR^d\mid\|x\|\leq 1\}$ and $\sphere^{d-1}:=\{x\in\RR^d\mid\|x\|=1\}$, respectively. A ball centred at the origin and of radius $r>0$ is denoted by $\ball_r^d$. The (Lebesgue) volume of $\ball^d$ and surface area of $\sphere^{d-1}$ are given by 
$$
\kappa_d:=\frac{\pi^{d/2}}{\Gamma(1+\frac{d}{2})}\qquad\text{and}\qquad \omega_d:=d\kappa_d=\frac{2\pi^{d/2}}{\Gamma(\frac{d}{2})}.
$$
The spherical Lebesgue measure on $\sphere^{d-1}$ is denoted by $\sigma$.
We use the convention to add the dimension of the sphere as an index when we consider the spherical Lebesgue measure on a lower dimensional unit sphere. For example we denote by $\sigma_{k-1}$ the spherical Lebesgue measure on the $(k-1)$-dimensional sphere $\sphere^{d-1}\cap F$, where $F$ is a $k$-dimensional linear subspace. We use the notation $\mu\mres A$ to indicate the restriction of a measure $\mu$ to a (measurable) set $A$.

For $k\in\{0,1,\ldots,d\}$, the linear (resp.\ affine) Grassmanian $G(d,k)$ (resp.\ $A(d,k)$) is the space of all $k$-dimensional linear (resp.\ affine) subspaces of $\RR^d$. These spaces are endowed with their usual topologies and Haar measures $\nu_k$ (resp.\ $\mu_k$), normalised as in \cite{schneider2008stochastic}.
Let $0 \leq k \leq l \leq m \leq d $ be integers, $E_o\in G(d,l)$ and $E\in A(d,l)$.
We introduce the following notation for spaces of linear or affine spaces which contain or are contained in $E_o$ or $E$:
\[ G(E_o,k) := \{ F_o \in G(d,k) \mid F_o \subset E_o \} , \]
\[ G(E_o,m) := \{ F_o \in G(d,m) \mid F_o \supset E_o \} , \]
\[ A(E,k) := \{ F \in A(d,k) \mid F \subset E \} , \]
\[ A(E,m) := \{ F \in A(d,m) \mid F \supset E \} . \]
These spaces are again equipped with the usual topologies and relative Haar measures $\nu_k^{E_o}$, $\nu_m^{E_o}$, $\mu_k^{E}$ and $\mu_m^{E}$, see \cite[Chapter 7.1]{schneider2008stochastic}.

For real-valued functions $f$ and $g$, we use the standard Landau notation $f(t) = O(g(t))$ which means that there exists a positive constant $C$ such that $\lvert f(t)\rvert \leq C \lvert g(t) \rvert $ for all $t$. 

The constants $C$ in this paper might depend on the space dimension $d$, but on nothing else. Their value may change from occurrence to occurrence.

\subsection{An integral-geometric transformation formula}

In the present article we often have to deal with multivariate integrals over the Grassmannians or relative Grassmannians we introduced in the previous section. Such integrals can usually be evaluated or simplified by applying suitable integral-geometric transformation formulas, which can be found in  \cite{schneider2008stochastic}, for example.

A special case of the affine Blaschke-Petkantschin formula \cite[Theorem 7.2.8]{schneider2008stochastic} says that for any $\ell\in\{1,\ldots,d\}$ and any non-negative measurable function $f\colon A(d,d-1)^{\ell} \to \RR$, we have
\begin{equation} \label{eq:B-P}
\begin{split}
    & \int_{A(d,d-1)^\ell} f(H_1,\ldots,H_\ell) \mu_{d-1}^{\otimes \ell} (\dint (H_1,\ldots,H_\ell))
    \\ & \qquad  = c_{d,\ell} \int_{A(d,d-\ell)} \int_{A(E,d-1)^\ell} f(H_1,\ldots,H_\ell) [H_1,\ldots,H_\ell]^{d-\ell+1} \\& \hspace{6cm} \times (\mu_{d-1}^{E})^{\otimes \ell} (\dint (H_1,\ldots,H_\ell)) \mu_{d-\ell}(\dint E) ,
\end{split}    
\end{equation}
where
\[
c_{d,\ell} := \frac{\omega_{d-\ell+1} \cdots \omega_d}{\omega_1 \cdots \omega_\ell} \Bigl( \frac{\omega_\ell}{\omega_d} \Bigr)^{\ell}
\]
is a constant depending only on $d$ and $\ell$, and $[H_1,\ldots,H_\ell]$ is the so-called \textit{subspace de\-ter\-mi\-nant}, which can be defined as the $\ell$-dimensional Lebesgue volume of the parallelepiped spanned by $u_1,\ldots,u_\ell$, where for each $i\in\{1,\ldots,\ell\}$, $u_i$ is one of the two unit normal vectors of $H_i$ (arbitrarily chosen).

As the definition of the subspace determinant above suggests, it is sometime more convenient to work with the normal vectors of hyperplances rather than the hyperplanes themselves.
In this spirit, the following formula is quite useful.
For any non-negative measurable function $f\colon G(d,d-1) \to \RR$, 
\begin{equation}\label{eq:July13a}
\begin{split}
    \int_{G(d,d-1)} f(H) \nu_{d-1}(\dint H)
    & = \frac{1}{\omega_d}\int_{\sphere^{d-1}} f(u^{\perp}) \sigma (\dint u) .
\end{split}    
\end{equation}

%\begin{figure}[t]
%    \centering
%    \includegraphics[width=0.3\textwidth]{ch5.png}\quad
%    \includegraphics[width=0.3\textwidth]{ch20.png}\quad
%    \includegraphics[width=0.3\textwidth]{ch50.png}
%    \caption{Simulations of convex hulls of the intersection point process $\Xi_{t,R}$ for $t=3500$ (left panel), $t=250\,000$ (middle panel) and $t=15\,000\,000$ (right panel).}
%    \label{fig:simulation2}
%\end{figure}

\subsection{Poisson hyperplane processes and derived point processes}

We denote by $\eta_t$, $t>0$, a stationary and isotropic Poisson hyperplane process of intensity $t$, i.e., a Poisson point process on $A(d,d-1)$ with intensity measure $t\mu_{d-1}$, see \cite{last2018lectures,schneider2008stochastic} for formal definitions.
Furthermore, we denote by 
$\eta_{t,R}$ the restriction of $\eta_{t}$ to the set of hyperplanes that intersect the ball $\ball_{R}^{d}$ of radius $R>0$.
Almost surely the hyperplanes of this process are in general position, meaning that the intersection of any $k$ of them has codimension $k$, for any $k\in\{1,\ldots,d\}$. In this paper, and in particular in the next definition, we assume that all realisations of the underlying hyperplane process satisfy this condition. We are interested in the intersection point process $\Xi_{t,R}$ induced by $\eta_{t,R}$, which is defined by
\begin{align} \label{intersection process1}
    \Xi_{t,R}(\,\cdot\,)
    := \frac{1}{d!}\sum_{(H_1,\ldots,H_d)\in \eta_{t,R,\neq}^{d}} \ind{H_1\cap \ldots \cap H_d  \in\, \cdot\, } ,
\end{align}
where $\eta_{t,R,\neq}^{d}$ is the set of all $d$-tuple of $d$ distinct hyperplanes of $\eta_{t,R}$.
%(as usual, we identify simple point processes with their supports) 
Note that almost surely the intersections in \eqref{intersection process1} are single points.
Applying the multivariate Mecke formula  \cite[Theorem 4.4]{last2018lectures} we see that the intensity measure of the intersection point process $\Xi_{t,R}$ is
\begin{equation} \label{Mecke formula}
    L_{t,R}(\,\cdot\,)
    = \frac{t^d}{d!}\int_{[\ball_{R}^{d}]^d}\ind{H_1\cap\ldots \cap H_d \in \cdot } \mu_{d-1}^{\otimes d} (\dint (H_1,\ldots,H_d)),
\end{equation}
where
\begin{equation*}
    [\ball_{R}^{d}]
    := \{ H \in A(d,d-1) \mid H\cap \ball_R^d \neq \emptyset \} 
\end{equation*}
is a shorthand notation for the set of hyperplanes intersecting the ball $\ball_R^d$.

\section{Preparatory lemmas}\label{sec:preplemmas}
%%%%%%%%%%%%%%%%%%%%%%%%%%%%%%%%%%%%%%%

\subsection{Intersection of a fixed affine flat with a random linear subspace}

The Cauchy distribution on $\RR$ with density 
$$
x\mapsto \frac{1}{\pi \gamma}\,\frac{1}{1+\big(\frac{x}{\gamma}\big)^2},\qquad x\in\RR,
$$
is known to arise as the distribution of the intersection point in the plane between the $x$-axis and a random line passing through the point with coordinates $(0,\gamma)$ and with direction uniformly distributed (Huygen's principle).
In this paper some computations involve a very similar quantity, namely the norm of the intersection point between a random $k$-dimensional linear space and a fixed affine flat of complementary dimension $d-k$.
The next lemma provides its distribution.
In this lemma we represent random linear subspaces (resp.\ flats) as fixed linear subspaces (resp.\ flats) on which we apply a random rotation.
\begin{lemma} \label{lem:intersection}
    Let $k\in\{1,\ldots,d-1\}$, $E_o\in G(d,k)$ and $F\in A(d,d-k)$ be fixed.
    Denote by $s_F=\mathrm{dist}(0,F)$ the distance between the origin and the flat $F$.
    Let $\theta \in SO_d$ be a uniformly distributed random rotation, i.e., a rotation distributed according to the Haar probability measure on the compact topological group $SO_d$.
    Then, denoting by $\overset{d}{=}$ equality in distribution, we have that
    \begin{equation*}
        \norm{E_o\cap \theta F} \overset{d}{=} \norm{\theta E_o \cap F},
    \end{equation*}
    and the density of both random variables is given by
    \begin{equation} \label{density-intersection}
        \ind{r\geq s_F} \frac{2\omega_{d-k}}{\omega_{d-k+1}} \Bigl( 1-\frac{s_F^2}{r^2} \Bigr)^{\frac{d-k-2}{2}} \frac{s_F}{r^2},\qquad r>0.
    \end{equation}
    In particular, for $r>0$,
    \begin{equation} \label{bigO-intersection}
        \PP(\norm{E_o\cap \theta F}\geq r)
        = \PP(\norm{\theta E_o \cap F}\geq r)
        = O\Bigl(\min\bigl(1,\frac{s_F}{r}\bigr)\Bigr).
    \end{equation}
\end{lemma}
\begin{proof}
    The first statement is trivial since 
    \[ \norm{E_o\cap \theta F} 
    = \norm{\theta^{-1} E_o \cap F} \]
    and $\theta^{-1}$ has the same distribution as $\theta$.
    %The last statement follow from the second one by integration.
    %Thus we only have to compute the density of the random variable $\norm{\theta E_o \cap F}$.
    
    Let us denote by $\widetilde{F} := \mathrm{span}(F) $ the $(d-k+1)$-dimensional linear space spanned by $F$.
    Now, we make two observations which will be used for a \enquote{dimension reduction argument}. First
    \begin{equation*}
        \theta E_o \cap F 
        = (\theta E_o\cap\widetilde{F}) \cap F,
    \end{equation*}
    and second, for any non-negative measurable function $f\colon G(\widetilde{F},1) \to \RR$, one has that
    \begin{equation*}
        \int_{SO_d} f(\theta E_o\cap\widetilde{F}) \nu (\dint \theta)
        = \int_{G(\widetilde{F},1)} f(L) \nu_1(\dint L)
        = \int_{\sphere^{d-1}\cap\widetilde{F}} f(\mathrm{span}(u)) \frac{\sigma_{d-k}(\dint u)}{\omega_{d-k+1}}  .
    \end{equation*}
    The first equality is a simple consequence of the rotation invariance of the Haar measures $\nu$ on $SO_d$ and $\nu_1$ on $G(\widetilde{F},1)$, and the fact that the intersection of linear spaces of dimensions $k$ and $d-k+1$ has dimension one, whenever the two spaces are in general position.
    The second equality is a parametrisation of $G(\widetilde{F},1)$.
    Therefore we can \enquote{reduce the ambient space} in the following way:
    \begin{equation*}
        \PP(\norm{\theta E_o \cap F} \in \,\cdot \,)
        = \int_{\sphere^{d-1}\cap\widetilde{F}} \ind{\norm{\mathrm{span}(u) \cap F )} \in \,\cdot\, } \frac{\sigma_{d-k}(\dint u)}{\omega_{d-k+1}} .
    \end{equation*}
    Decompose the affine flat $F$ into a sum $x_F+ F_o$, where $F_o\in G(d,d-k)$ is translated by a vector $x_F\in F_o^{\perp}$.
    An elementary geometric analysis in the right-angled triangle with vertices the origin, $x_F$ and ${\mathrm{span}(u) \cap F}$ leads to
    \begin{equation*}
        \norm{\mathrm{span}(u) \cap F)}
        = \frac{\norm{x_F}^2}{\lvert \langle u , x_F \rangle \rvert} ,
    \end{equation*}
    and therefore we obtain
    \begin{equation*}
        \PP(\norm{\theta E_o \cap F} \geq r )
        = \int_{\sphere^{d-1}\cap\widetilde{F}} \ind{ \bigl\lvert \bigl\langle u , \frac{x_F}{\norm{x_F}} \bigr\rangle \bigr\rvert \leq \frac{\norm{x_F}}{r} } \frac{\sigma_{d-k}(\dint u)}{\omega_{d-k+1}} .
    \end{equation*}
    This is the normalised $(d-k)$-dimensional volume of a certain neighborhood of the great subsphere $\sphere^{d-1}\cap \widetilde{F} \cap x_F^\perp$ of the sphere  $\sphere^{d-1}\cap \widetilde{F}$.
    If $r\leq \norm{x_F}$, this neighborhood is the full sphere and thus the probability is $1$ in that case. 
    Otherwise, it is some kind of \enquote{slightly bent} cylinder of height $\frac{\norm{x_F}}{r}$. % and it is easy to see that $K_{1,1} = \frac{2 \norm{x_F}}{r} (1+o(1)) $. 
    With the slice integration formula \cite[Corollary A.5]{axler1992harmonic} 
    %\footnote{In this reference they use the notation $\spheren$ for what we denote $\sphere^{n-1}$ (or equivalently $\sphere^{d-1}\cap F$ where $F$ is a $n$-dimensional linear space) and $\sigma_n$ for the \textit{normalized} spherical Lebesgue measure on $S_n$, which we denote by $\frac{1}{\omega_n} \sigma_{n-1}$. Here we have $n=\dim \widetilde{F} = d-k+1$.}
    this is
    \begin{align*}
        \PP(\norm{\theta E_o \cap F} \geq r )
        &= \frac{\omega_{d-k}}{\omega_{d-k+1}} \int_{-1}^1 \ind{ \lvert y \rvert \leq \frac{\norm{x_F}}{r} } (1-y^2)^{\frac{d-k-2}{2}} \dint y .
    \end{align*}
%    Using a linear substitution and the fact that the integrand is an even function, we rewrite this as
%    \begin{align*}
%        \PP(\norm{\theta E_o \cap F} \geq r )
%        %&= \frac{2\omega_{d-k}}{\omega_{d-k+1}} \int_0^1 \ind{ r y \leq \norm{x_F} } (1-y^2)^{\frac{d-k-2}{2}} \dint y 
%        %\\&= \frac{2\omega_{d-k}}{\omega_{d-k+1}} \int_0^r \ind{ z \leq \norm{x_F} } \Bigl(1-\Bigl(\frac{z}{r}\Bigr)^2 \Bigr)^{\frac{d-k-2}{2}} \frac{\dint z}{r} \\
%        &= \frac{2\omega_{d-k}}{\omega_{d-k+1}} \int_0^{\min(r,\norm{x_F})} \Bigl(1-\Bigl(\frac{z}{r} \Bigr)^2 \Bigr)^{\frac{d-k-2}{2}} \frac{\dint z}{r} .
%    \end{align*}
    Taking the derivative with respect to $r$ of the last expression provides the density \eqref{density-intersection} and bounding $1-y^2$ by $1$ proves the estimate \eqref{bigO-intersection}.
\end{proof}

\subsection{Another preparation}

The next lemma is a significant ingredient for computation in the next section. 
\begin{lemma}\label{density}
    Let $d \geq 2$, $k \in \{0,\ldots,d-2\}$, $a \geq k+1$.
    Consider the function $J_{d,k,a}: A(d,k) \times [0,\infty) \to [0, \infty)$ defined by 
    \begin{equation*}
    \begin{split}
        &J_{d,k,a}(E,R):=\int_{(A(E,d-1)\cap [\ball_{R}^d])^{d-k}}[H_{1},\ldots,H_{d-k}]^a (\mu_{d-1}^{E})^{\otimes (d-k)} (\dint (H_{1}, \ldots, H_{d-k})) .
    \end{split}
    \end{equation*}
    Then, if we denote by $s_E$ the distance between the origin a flat $E\in A(d,k)$, we have that 
    \begin{enumerate}
        \item the value $J_{d,k,a}(E,R)$ is a function of the ratio $R/s_E$,
        \item and 
        \begin{align*}
            J_{d,k,a}(E,R) 
            &= \ind{1<R/s_E} C_{d-k,a}^{(1)} 
             + \ind{1\geq R/s_E} \Bigl(\frac{R}{s_E}\Bigr)^{d-k+a} \Bigl[ C_{d-k,a}^{(2)} + O \Bigl( \frac{R^2}{s_E^2}\Bigr)  \Bigr] ,
        \end{align*}
        where the constants $C_{d-k,a}^{(i)}$, $i=1,2$, depend only on the dimension difference $d-k$ and $a$, and the constant involved in bounding the big $O$ term depends only on $d$,  $k$ and $a$.
    \end{enumerate}
\end{lemma}
\begin{remark}
    The constants $C_{d-k,a}^{(1)}$ and $C_{d-k,a}^{(2)}$ are represented in the proof below by the integrals \eqref{constantC1} and \eqref{eq:DefConst} respectively.
\end{remark}
\begin{proof} 
    In this proof we will drop the indices $d$, $k$ and $a$ appearing in $J_{d,k,a}$ and write only $J$ instead.
    Also we define $E_o\in G(d,k)$ and $x_E\in E$ to be, respectively, the $k$-dimensional linear subspace parallel to $E$ and the orthogonal projection of the origin onto $E$. In particular $E=E_o+x_E$ and $s_E=\norm{x_E}$.
    
    Let $\theta \in SO_d$ be an arbitrary rotation and $\alpha>0$.
    By considering the substitutions $\widetilde{H_i} = \alpha \theta H_i$, $i\in\{1,\ldots,d-k\}$, we see  that $J(E_o+ x_E, R) = J(\theta E_o+\alpha \theta x_E,\alpha R)$.
    In particular, by taking $\alpha=s_E^{-1}=\norm{x_E}^{-1}$, we see that the first claim of the lemma holds.
    
    The rest of the proof is dedicated to show the second claim.
    First, we represent the affine hyperplanes $H_i$, $i\in\{1,\ldots,d-k\}$, as sums $L_i+x_E$, where the $L_i$'s are the linear hyperplanes parallel to the $H_i$'s.
    Note that through this rewriting we can replace the subspace determinant in the definition of $J$ by $[L_{1},\ldots,L_{d-k}]$, since it is invariant under translations of its components.
    Also we recall that, by definition, $\mu_{d-1}^{E}(\dint H_{i}) = \nu_{d-1}^{E_{o}}(\dint L_{1}) $.
    We get
    \begin{align} \label{eq:J1}
        J(E_{o}+x_{E},R)
        &= \int_{(G(E_{o},d-1))^{d-k}}\prod_{i=1}^{d-k} \ind{L_{i}+x_{E} \in [\ball_{R}^d]} 
        \\&\hspace{3cm} \times [L_{1},\ldots,L_{d-k}]^a (\nu_{d-1}^{E_{o}})^{\otimes (d-k)} (\dint (L_{1}, \ldots, L_{d-k})).
        \notag
    \end{align}
    Next, we parametrise the hyperplanes $L_i$ by their unit orthonormal vectors $u_i$. By duality these vectors belong to $E_o^\perp$. Indeed, for any of the $u_i$'s, the corresponding orthogonal hyperplanes $u_i^\perp = L_i$ contain $E_o$. More precisely, we use the transformation
    \[ \int_{G(E_o,d-1)} f(L) \nu_{d-1}^{E_o} (\dint L)
    = \int_{ \sphere^{d-1} \cap E_o^\perp } f(u^\perp) \frac{\sigma_{d-k-1} (\dint u)}{\omega_{d-k}} , \]
    which holds for any non-negative measurable function $f\colon G(E_o,d-1) \to \RR$ and follows from the invariance and uniqueness of the Haar measure on $G(E_o,d-1)$ similarly to \eqref{eq:July13a}.
    Iterating this formula over the $(d-k)$-fold integral in \eqref{eq:J1} and using the fact that 
    \[ [u_1^\perp,\ldots,u_{d-k}^\perp] = [u_1,\ldots,u_{d-k}], \] we get
    \begin{align*}
        J(E_{o}+x_{E},R)
        &=\frac{1}{(\omega_{d-k})^{d-k}} \int_{(\sphere^{d-1}\cap E_o^\perp)^{d-k}} \prod_{i=1}^{d-k} \ind{u_i^\perp+x_{E} \in [\ball_{R}^d]} 
        \\&\hspace{3cm} \times [u_{1},\ldots,u_{d-k}]^a \sigma_{d-k-1}^{\otimes (d-k)} (\dint (u_{1}, \ldots, u_{d-k})).
    \end{align*}
    
    Now, we will simplify the indicator functions.
    For this we observe that a hyperplane of the form $u^\perp + x_E$ can be written as $ u^\perp + \langle u , x_E \rangle u $. Indeed, these two expressions represent the same hyperplane characterised by being parallel to $u^\perp$ and containing the point $x_E$.
    It makes clear that such a hyperplane intersects the ball $\ball_R^d$ if and only if the scalar product $\langle u , x_E \rangle$ is between $-R$ and $R$.
    Therefore,
    \begin{align} \label{eq:J2}
        J(E_{o}+x_{E},R)
        &=\frac{1}{(\omega_{d-k})^{d-k}} \int_{(\sphere^{d-1}\cap E_o^\perp)^{d-k}} \prod_{i=1}^{d-k} \ind{ | \langle u_i , x_E\rangle | \leq R} 
        \\&\hspace{3cm} \times [u_{1},\ldots,u_{d-k}]^a \sigma_{d-k-1}^{\otimes (d-k)} (\dint (u_{1}, \ldots, u_{d-k})).
        \notag
    \end{align}
    At this point, we split the proof by considering whether $\norm{x_E} < R$ (easy part) or $\norm{x_E} \geq R$ (more intricate part).
    
    \paragraph{Case where $\norm{x_E} < R$ :} Thanks to the Cauchy-Schwarz inequality we see that, in this case, the conditions in the indicator functions in \eqref{eq:J2} are always satisfied and therefore
    \begin{align} \label{constantC1}
        &J(E_{o}+x_{E},R)=\frac{1}{(\omega_{d-k})^{d-k}} \int_{(\sphere^{d-1}\cap E_o^\perp)^{d-k}} [u_{1},\ldots,u_{d-k}]^a \sigma_{d-k-1}^{\otimes (d-k-1)} (\dint (u_{1}, \ldots, u_{d-k})).
    \end{align}
    The later expression is a constant which depends only on $a$ and $d-k$, and therefore the lemma is proved for the case where $\norm{x_E} < R$.
    %\footnote{Moreover we could compute this constant explicitly. Theorem 8.2.2 in \cite{schneider2008stochastic} gives the moments of $[x_{1},\ldots,x_{d-k}]$ for points uniformly distributed in the ball (!!) by using polar coordinates this can easily be related to our constant. \textcolor{blue}{Do we want to do that ?}}
    
    \paragraph{Case where $\norm{x_E} \geq R$ :}
    
    Now we will adapt the slice integration formula \cite[Corollary A.5]{axler1992harmonic} to transform integrals over the $(d-k-1)$-dimensional unit sphere $\sphere^{d-1}\cap E_o^\perp$ of the space $E_o^\perp$ into integrals over the $(d-k-2)$-dimensional slices obtained by cutting that sphere by hyperplanes parallel to $x_E^\perp$.
    In this set-up, the slice integration formula can be written as 
    \begin{align*}
        \int_{\sphere^{d-1}\cap E_o^\perp} f(u) \sigma_{d-k}(\dint u)
        &= \int_{-1}^1 (1-y^2)^{\frac{(d-k-3) }{2}} \int_{(\sphere^{d-1}\cap E_{o}^{\perp})\cap x_{E}^{\perp}}
        \\&\qquad\qquad \times f\left( \sqrt{1-y^2} u'+y\frac{x_{E}}{\norm{x_E}} \right) \sigma_{d-k-2} (\dint u') \, \dint y ,
    \end{align*}
    for any non-negative measurable function $f\colon \sphere^{d-1}\cap E_o^\perp \to \RR$.
    Iterating this formula over the $(d-k)$-fold integral in \eqref{eq:J2}, we get
    \begin{align*}
        J(E_{o}+x_{E},R)
        &=\frac{1}{(\omega_{d-k})^{d-k}} \int_{[-1,1]^{d-k}} \prod_{i=1}^{d-k} (1-y_i^2)^{\frac{(d-k-3) }{2}} \int_{((\sphere^{d-1}\cap E_o^\perp)\cap x_E^\perp)^{d-k}} 
        \\&\quad \times \prod_{i=1}^{d-k} \ind{ \Bigl| \Bigl\langle \sqrt{1-y_i^2} u'_i+y_i\frac{x_{E}}{\norm{x_E}} , x_E \Bigr\rangle \Bigr| \leq R} 
        \\&\quad \times \Bigl[\sqrt{1-y_1^2} u'_1+y_1\frac{x_{E}}{\norm{x_E}},\ldots,\sqrt{1-y_i^2} u'_{d-k}+y_{d-k}\frac{x_{E}}{\norm{x_E}} \Bigr]^a 
        \\&\quad \times \sigma_{d-k-2}^{\otimes (d-k)} (\dint (u'_{1}, \ldots, u'_{d-k})) \, \dint (y_1,\ldots,y_{d-k}) .
        \notag
    \end{align*}    
    We will now simplify the condition in the indicator functions.
    Using that, for any $i\in\{1,\ldots,d-k\}$, the vectors $u'_i$ and $x_E$ are orthogonal, we see that
    \[ \Bigl\langle \sqrt{1-y_i^2} u'_i+y_i\frac{x_{E}}{\norm{x_E}} , x_E \Bigr\rangle
    = \Bigl\langle y_i\frac{x_{E}}{\norm{x_E}} , x_E \Bigr\rangle
    = y_i \, \norm{x_E} , \]
    and thus we get the simplification
    \begin{align*}
        \prod_{i=1}^{d-k} \ind{ \Bigl| \Bigl\langle \sqrt{1-y_i^2} u'_i+y_i\frac{x_{E}}{\norm{x_E}} , x_E \Bigr\rangle \Bigr| \leq R} 
        = \prod_{i=1}^{d-k} \ind{ |y_i| \leq \frac{R}{\norm{x_E}} } .
    \end{align*}
    Since we are in the case where $\norm{x_E} \geq R$, we have that the interval $[- \frac{R}{\norm{x_E}} , \frac{R}{\norm{x_E}} ]$ is a subset of $[-1,1]$ and therefore we can pass the lately discussed conditions into the domain of the outer integral as follows:
    \begin{align*}
        J(E_{o}+x_{E},R)
        &=\frac{1}{(\omega_{d-k})^{d-k}} \int_{\bigl[ -\frac{R}{\norm{x_E}} , \frac{R}{\norm{x_E}} \bigr]^{d-k}} \prod_{i=1}^{d-k} (1-y_i^2)^{\frac{(d-k-3) }{2}} \int_{((\sphere^{d-1}\cap E_o^\perp)\cap x_E^\perp)^{d-k}} 
        \\&\quad \times \Bigl[\sqrt{1-y_1^2} u'_1+y_1\frac{x_{E}}{\norm{x_E}},\ldots,\sqrt{1-y_i^2} u'_{d-k}+y_{d-k}\frac{x_{E}}{\norm{x_E}} \Bigr]^a 
        \\&\quad \times \sigma_{d-k-2}^{\otimes (d-k)} (\dint (u'_{1}, \ldots, u'_{d-k})) \, \dint (y_1,\ldots,y_{d-k}) .
        \notag
    \end{align*}  
 In the next step we apply for each $i\in\{1,\ldots,d-k\}$ the linear substitution $z_i = \frac{\norm{x_E}}{R}y_i$ which have the effect of bringing the integration domain of the outer integral back to the unit cube $[-1,1]^{d-k}$, i.e.,
    \begin{align*}
        &J(E_{o}+x_{E},R)
        \\&=\Bigl(\frac{R}{\omega_{d-k} \, \norm{x_E}}\Bigr)^{d-k} \int_{[-1,1]^{d-k}} \prod_{i=1}^{d-k} \Bigl(1- \frac{R^2}{\norm{x_E}^2}z_i^2 \Bigr)^{\frac{(d-k-3) }{2}} \int_{((\sphere^{d-1}\cap E_o^\perp)\cap x_E^\perp)^{d-k}} 
        \\&\quad \times \Bigl[\sqrt{1- \Bigl(\frac{R z_1}{\norm{x_E}}\Bigr)^2} u'_1+\frac{R z_1}{\norm{x_E}} \frac{x_{E}}{\norm{x_E}} ,\ldots, \sqrt{1- \Bigl(\frac{R z_{d-k}}{\norm{x_E}}\Bigr)^2} u'_{d-k}+\frac{R z_{d-k}}{\norm{x_E}} \frac{x_{E}}{\norm{x_E}} \Bigr]^a 
        \\&\quad \times \sigma_{d-k-2}^{\otimes (d-k)} (\dint (u'_{1}, \ldots, u'_{d-k})) \, \dint (z_1,\ldots,z_{d-k}) .
        \notag
    \end{align*}  
     Using that the vectors $u'_i$'s are orthogonal to $x_E$ and multilinearity we can take out of the subspace determinant the factor $\frac{R}{\norm{x_E}}$, which appears in front of each term $z_i \frac{x_E}{\norm{x_E}}$. This gives the equality
    \begin{align*}
        &J(E_{o}+x_{E},R)
        \\&=\frac{1}{(\omega_{d-k})^{d-k}} \Bigl(\frac{R}{\norm{x_E}}\Bigr)^{d-k+a} \int_{[-1,1]^{d-k}} \prod_{i=1}^{d-k} \Bigl(1- \frac{R^2}{\norm{x_E}^2}z_i^2 \Bigr)^{\frac{(d-k-3) }{2}} \int_{((\sphere^{d-1}\cap E_o^\perp)\cap x_E^\perp)^{d-k}}
        \\&\quad \times \Bigl[\sqrt{1- \Bigl(\frac{R z_1}{\norm{x_E}}\Bigr)^2} u'_1+z_1 \frac{x_{E}}{\norm{x_E}} ,\ldots, \sqrt{1- \Bigl(\frac{R z_{d-k}}{\norm{x_E}}\Bigr)^2} u'_{d-k}+z_{d-k} \frac{x_{E}}{\norm{x_E}} \Bigr]^a 
        \\&\quad \times \sigma_{d-k-2}^{\otimes (d-k)} (\dint (u'_{1}, \ldots, u'_{d-k})) \, \dint (z_1,\ldots,z_{d-k}) .
    \end{align*}  
    
    In the final part of this proof we will approximate the product and subspace determinant above by terms independent from $R$ and $\norm{x_E}$, and bound the approximation error in terms of the ratio $\frac{R}{\norm{x_E}}$. This will eventually allow us to conclude the proof.
    
    Since in the last integral the variables $z_i$ all belong  to the bounded interval $[-1,1]$, we have that the product is arbitrarily close to $1$ as long as the ratio $\frac{R}{\norm{x_E}}$ is small enough. More precisely one can check that
    \begin{align*}
        \prod_{i=1}^{d-k} \Bigl(1- \frac{R^2}{\norm{x_E}^2}z_i^2 \Bigr)^{\frac{(d-k-3) }{2}}
        = 1 - O \Bigl( \frac{R^2}{\norm{x_E}^2} \Bigr) ,
    \end{align*}
    where the constant in the big $O$ term is a bounded positive number depending only on $d$ and $k$. %by expending the product we can take this constant optimally and that would be $(d-k)(d-k-3)/2$.
    
    Similarly as above we have that, for $z_i\in[-1,1]$,
    \[ \sqrt{1- \Bigl(\frac{R z_i}{\norm{x_E}}\Bigr)^2} = 1 - O \Bigl( \frac{R^2}{\norm{x_E}^2} \Bigr) . \]
    In particular, for any given $u'_i$'s and $z_i$'s the subspace determinant in the last integral tends to 
    \( \Bigl[ u'_1+z_1 \frac{x_{E}}{\norm{x_E}} ,\ldots, u'_{d-k}+z_{d-k} \frac{x_{E}}{\norm{x_E}} \Bigr] \), as the ratio $\frac{R}{\norm{x_E}}$ goes to $0$. We still need to bound the error involved in this approximation.
    To do this we observe that the (subspace) determinant is a locally Lipschitz function and that each of the involved vectors
    \( \sqrt{1- \Bigl(\frac{R z_i}{\norm{x_E}}\Bigr)^2} u'_i+z_i \frac{x_{E}}{\norm{x_E}} \)
    belong to the bounded cylinder 
    \( ((\sphere^{d-1}\cap E_o^\perp)\cap x_E^\perp) + [-1,1] \frac{x_E}{\norm{x_E}}  \).
    Therefore,
    \begin{align*}
        & \Bigl[\sqrt{1- \Bigl(\frac{R z_1}{\norm{x_E}}\Bigr)^2} u'_1+z_1 \frac{x_{E}}{\norm{x_E}} ,\ldots, \sqrt{1- \Bigl(\frac{R z_{d-k}}{\norm{x_E}}\Bigr)^2} u'_{d-k}+z_{d-k} \frac{x_{E}}{\norm{x_E}} \Bigr]
        \\& = \Bigl[ u'_1+z_1 \frac{x_{E}}{\norm{x_E}} ,\ldots,  u'_{d-k}+z_{d-k} \frac{x_{E}}{\norm{x_E}} \Bigr] + O \Bigl( \frac{R^2}{\norm{x_E}^2}\Bigr) .
    \end{align*}
    Putting these approximations together and remembering that the constants in the big $O$ terms depend only on $d$ and $k$, we can take these error terms out of the integral of investigation, which gives 
    \begin{align*}
        J(E_{o}+x_{E},R)
        &= \Bigl(\frac{R}{\norm{x_E}}\Bigr)^{d-k+a} \Bigl[ C_{d-k,a}^{(2)} + O \Bigl( \frac{R^2}{\norm{x_E}^2}\Bigr)  \Bigr],
    \end{align*}     
    where the constants involved in bounding the big $O$ term depend now on $a$ (and also $d$ and $k$ as before), and $C_{d-k,a}^{(2)}$ is the constant, which depends only on $d-k$ and $a$ and is defined by
    \begin{equation}\label{eq:DefConst}
    \begin{split}
        C_{d-k,a}^{(2)}
        &:= \frac{1}{(\omega_{d-k})^{d-k}} \int_{[-1,1]^{d-k}} \int_{((\sphere^{d-1}\cap E_o^\perp)\cap x_E^\perp)^{d-k}} 
         \Bigl[ u'_1+z_1 \frac{x_{E}}{\norm{x_E}} ,\ldots
        \\&\hspace{1cm} \ldots, u'_{d-k}+z_{d-k} \frac{x_{E}}{\norm{x_E}} \Bigr]^a \sigma_{d-k-2}^{\otimes (d-k)} (\dint (u'_{1}, \ldots, u'_{d-k})) \, \dint (z_1,\ldots,z_{d-k}) .
        \end{split}    
    \end{equation}
    This concludes the proof.
\end{proof}

%%%%%%%%%%%%%%%%%%%%%%%%%%%%%%%
\section{Convergence of the intensity measure}\label{sec:ConvIntensityMeasure}
%%%%%%%%%%%%%%%%%%%%%%%%%%%%%%%

%%%%%%%
\subsection{Pointwise convergence of the density function}
%%%%%%%

In this section we consider the convergence of the intensity measure $L_{t,R}$ of the intersection point process $\Xi_{t,R}$, as $t\to\infty$. We start with a lemma dealing with the convergence of the density function $f_{t,R}$ of $L_{t,R}$.

\begin{lemma} \label{densityofintensitymeasure}
    The density $f_{t,R}$ of the intensity measure $L_{t,R}$, with respect to the Lebesgue measure on $\RR^d\setminus \{0\}$, satisfies
    \begin{align} 
        f_{t,R}(x)
        &= \ind{\norm{x}<R} \frac{C_{d,1}^{(1)}}{d!} t^d
        + \ind{\norm{x}\geq R} \Bigl(\frac{R}{\norm{x}}\Bigr)^{d+1} t^d \Bigl[ \frac{C_{d,1}^{(2)}}{d!} + O \Bigl( \frac{R^2}{\norm{x}^2}\Bigr)  \Bigr] ,
    \end{align}
    where the constants $C_{d,1}^{(i)}$, $i\in\{1,2\}$, are the same constants as in Lemma \ref{density} applied with $k=0$ and $a=1$.
    In particular, if $R=t^{-\frac{d}{d+1}}$, it follows immediately that
%    \begin{align*}
%        f_{t,R}(x)
%        &= \ind{\norm{x}<t^{-\frac{d}{d+1}}} C_{d,1}^{(1)} t^d 
%        + \ind{\norm{x}\geq t^{-\frac{d}{d+1}}} \Bigl(\frac{1}{\norm{x}}\Bigr)^{d+1} \Bigl[ C_{d,1}^{(2)} + O \Bigl( \frac{t^{-\frac{2d}{d+1}}}{\norm{x}^2}\Bigr)  \Bigr] ,
%    \end{align*}
    for any fixed $x\in\RR^d\setminus\{0\}$, 
    \begin{align*}
        \lim_{t\to\infty} f_{t,R}(x) 
        &= C_d \, \norm{x}^{-(d+1)} ,
    \end{align*}
    with
    %\begin{equation}%\label{eq:DefConstCd}
    %\begin{split}
    %    C_d
    %    &:= \frac{1}{d! \, (\omega_{d-k})^{d-k}} \int_{[-1,1]^{d-k}} \int_{((\sphere^{d-1}\cap E_o^\perp)\cap x_E^\perp)^{d-k}} 
    %     \Bigl[ u'_1+z_1 \frac{x_{E}}{\norm{x_E}} ,\ldots
    %    \\&\hspace{1cm} \ldots, u'_{d-k}+z_{d-k} \frac{x_{E}}{\norm{x_E}} \Bigr]^a \sigma_{d-k-2}^{\otimes (d-k)} (\dint (u'_{1}, \ldots, u'_{d-k})) \, \dint (z_1,\ldots,z_{d-k}) .
    %    \end{split}    
    %\end{equation}
    \begin{equation}\label{eq:DefConstCd}
    \begin{split}
        C_d
        &:= \frac{1}{d! \, (\omega_{d})^{d}} \int_{[-1,1]^{d}} \int_{((\sphere^{d-1}\cap e_{d}^\perp)^{d}} 
         \Bigl[ u_1+z_1 e_{d} ,\ldots, u_{d}+z_{d} e_{d} \Bigr] 
        \sigma_{d-2}^{\otimes d} (\dint (u_{1}, \ldots, u_{d}))
        \\&\hspace{10cm} \times \dint (z_1,\ldots,z_{d}) .
        \end{split}    
    \end{equation}
\end{lemma}
\begin{proof}
    Applying the Blaschke-Petkantschin formula \eqref{eq:B-P} to the integral representation \eqref{Mecke formula} of the intensity measure we get, for any Borel set $A\subset\RR^d\setminus\{0\}$,
    \begin{align}
        L_{t,R}(A) 
        &= \int_{A} \frac{t^d}{d!} \int_{(A(x,d-1) \cap [\ball_{R}^d])^d} [H_{1},\ldots,H_{d}] \mu_{d-1}^{x}(\dint(H_{1},\ldots,H_{d})) \dint x ,
    \end{align}
    and therefore the density function $f_{t,R}$ of $L_{t,R}$ is given by
    \[ f_{t,R}(x) = \frac{t^d}{d!} \int_{(A(x,d-1) \cap [\ball_{R}^d])^d} [H_{1},\ldots,H_{d}] \mu_{d-1}^{x}(\dint(H_{1},\ldots,H_{d})),\qquad x\in\RR^d\setminus\{0\}. \]
    Applying Lemma \ref{density} with $a=1$ and $k=0$ yields the result.
\end{proof}

%%%%%%%
\subsection{Convergence in total variation of the (restricted) intensity measure}
%%%%%%%

Recall that $L_{t,R}$ is the intensity measure of the intersection point process $\Xi_{t,R}$ where $R=t^{-\frac{d}{d+1}}$ and $M$ is the intensity measure of the limiting Poisson point process $\zeta$ which has density $C_d \norm{x}^{-(d+1)}$, where $C_d$ is the constant defined in \eqref{eq:DefConstCd}. Also, for $r>0$, we denote by $L_{t,R}\mres{(\ball_{r}^d)^c}$, $M\mres{(\ball_{r}^d)^c}$ the restrictions of these measures to the complement of the ball $\ball_r^d$.
The total variation distance between the two (restricted) measures is defined as
\begin{equation*}
    d_{\rm TV}(L_{t,R}\mres{(\ball_{r}^d)^c},M\mres{(\ball_{r}^d)^c})
    := \sup \{ \lvert L_{t,R}(A) - M(A) \rvert  \} ,
\end{equation*}
where the supremum is taken over all Borel set $A\subset(\ball_r^d)^c:=\RR^d\setminus\ball_{r}^d$.
As a simple corollary of Lemma \ref{densityofintensitymeasure} we find that for any fixed $r$ this distance goes to zero as the intensity goes to infinity.

\begin{corollary} \label{totalvariation}
    Let $r>0$.
    Assume that $R=t^{-\frac{d}{d+1}} < r$. Then there is a constant $C>0$ depending only on $d$ such that
    \[ d_{\rm TV}(L_{t,R}\mres{(\ball_{r}^d)^c},M\mres{(\ball_{r}^d)^c}) \leq C t^{-\frac{2d}{d+1}} r^{-3}. \]
    In particular,
    \[ d_{\rm TV}(L_{t,R}\mres{(\ball_{r}^d)^c},M\mres{(\ball_{r}^d)^c}) \to 0,\quad \text{ as } t\to \infty. \]
\end{corollary}
\begin{proof}
    Let $A\subset (\ball_{r}^d)^c$ be a Borel set.
    From Lemma \ref{densityofintensitymeasure} we have that
    \begin{align*}
        L_{t,R}(A)
        &= \int_A \norm{x}^{-d-1} \Bigl[ \frac{C_{d,1}^{(2)}}{d!} + O \Bigl( \frac{R^2}{\norm{x}^2}\Bigr)  \Bigr] \dint x
        = M(A) + \int_A \norm{x}^{-d-1} O \Bigl( \frac{R^2}{\norm{x}^2} \Bigr)  \dint x .
    \end{align*}
    Thus
    \begin{align*}
        \lvert L_{t,R}(A) - M(A) \rvert
        %= R^2 O \left( \int_{A} \norm{x}^{-d-3} \dint x \right)
        \leq C R^2 \int_{(\ball_{r}^d)^c} \norm{x}^{-d-3} \dint x 
        = C R^2 r^{-3} 
        = C t^{-\frac{2d}{d+1}} r^{-3}.
    \end{align*}
    Since this bound is independent from the choice of the set $A$, by taking the supremum over all $A\subset (\ball_{r}^d)^c$ we bound the total variation distance by a big $O$ of $R^2$, which goes to zero. 
    The proof is thus complete.
\end{proof}

%%%%%%%%%%%%%%%%%%%%%%%%%%%%%
\section{Convergence of the point process and its convex hull}\label{sec:ConvPPandConv}
%%%%%%%%%%%%%%%%%%%%%%%%%%%%%

%%%%%%%%%%%%%
\subsection{Convergence in Kantorovich–Rubinstein distance of the (restricted) intersection process}
%%%%%%%%%%%%%

In this section we show that the restriction of the intersection process $\Xi_{t,R}$ to the complement of a ball $\ball_r^d$ with radius $r>0$ converges in Kantorovich–Rubinstein distance to the restriction of the Poisson point process $\zeta$. Let us first introduce, with a geometric point of view, the Kantorovich–Rubinstein distance between simple point processes in $\RR^d$. For a definition applying to a more general setting we refer the reader to Section 2.5 of \cite{decreusefond2016functional} where this distance is introduced with a functional point of view. For two discrete point sets $S_1$ and $S_2$ in $\RR^d$ we say that the total variation distance between them is the quantity 
\[ d_{\mathrm{TV}} (S_1,S_2) := \max \bigl( \lvert S_1\setminus S_2 \rvert , \lvert S_2\setminus S_1 \rvert \bigr) , \]
where $\lvert S_i\setminus S_j \rvert$ denotes the number of points of $S_i$ which does not belong to $S_j$, possibly infininity. Note that this definition corresponds to the one used in the previous section when the sets $S_1$ and $S_2$ are considered as counting measures.
Now, consider two simple point processes $X_1$ and $X_2$ in $\RR^d$. Here, as in the rest of the article, we identify a simple point process with its support, which is a random discrete subset of $\RR^d$. The Kantorovich–Rubinstein distance between $X_1$ and $X_2$ is defined as
\[ d_{\mathrm{KR}} (X_1,X_2) := \inf_{(Y_1,Y_2)\in\Sigma (X_1,X_2)} \EE d_{\mathrm{TV}} (Y_1,Y_2) , \]
where $\Sigma (X_1,X_2)$ denotes the set of couplings of $X_1$ and $X_2$, i.e., the set of pairs $(Y_1,Y_2)$ such that $X_i$ and $Y_i$ have the same distribution, $i\in\{1,2\}$. Informally speaking, the Kantorovich–Rubinstein distance measures by how many points $X_1$ and $X_2$ differ on average, if we couple them optimally.

We want to use this distance to compare the intersection point process $\Xi_{t,R}$ with the Poisson point process $\zeta$. The first consists almost surely of finitely many points, while the second is almost surely infinite. Therefore their Kantorovich–Rubinstein distance is always infinite for any $t$.
Thus, in order to obtain a meaningful result we will consider the restriction of both point processes to the complement of a ball with positive radius. Doing so, we ensure that both point processes are almost surely finite.

\begin{theorem} \label{convKR}
    Assume that $0<R=t^{-d/(d+1)}<r<1$. Then, we have that
    \begin{equation*}
        d_{\mathrm{KR}}(\Xi_{t,R}\mres{(\ball_{r}^{d})^c}, \zeta\mres{(\ball_{r}^d})^c )
        \leq C t^{-\frac{1}{d+1}} \ln(t) r^{-3} ,
    \end{equation*}
    where $C$ is a positive constant which depends only on $d$.
    In particular
    \begin{equation}\label{eq:July13b}
        \lim_{t \to \infty} d_{\mathrm{KR}}(\Xi_{t,R}\mres{(\ball_{r}^{d})^c}, \zeta\mres{(\ball_{r}^d})^c ) = 0.
    \end{equation}
\end{theorem}
\begin{proof}
    Theorem 3.1 of \cite{decreusefond2016functional} applied to our setting says that
    \[ d_{\mathrm{KR}}(\Xi_{t,R}\mres{(\ball_{r}^{d})^c}, \zeta\mres{(\ball_{r}^d})^c )
    \leq d_{\rm TV}(L_{t,R}\mres{(\ball_{r}^d)^c},M\mres{(\ball_{r}^d)^c}) + \frac{2^{d+1}}{d!} \rho_{t,R}(r) , \]
    where $\rho_{t,R}(r)$ is defined (and bounded) in the Lemma \ref{errorfunction} below.
    Corollary \ref{totalvariation} and Lemma \ref{errorfunction} provide bounds for each of the two summands of the right hand side, and thus we get the bound
    \[ d_{\mathrm{KR}}(\Xi_{t,R}\mres{(\ball_{r}^{d})^c}, \zeta\mres{(\ball_{r}^d})^c )
        \leq C t^{\frac{-2d}{d+1}} r^{-3} + C t^{\frac{-1}{d+1}} \ln(t) r^{-2}  \]
    from which Theorem \ref{convKR} follows.
    Note that the assumptions $t^{-d/(d+1)}<r$ of Corollary \ref{totalvariation} and $t\geq 1$ of Lemma \ref{errorfunction} are both satisfied.
    Therefore the proof is complete.
\end{proof}

In the next lemma we deal with the error term $\rho_{t,R}$, which appeared in \eqref{eq:July13b} in the proof of the previous theorem.

\begin{lemma} \label{errorfunction}
    For $d \geq 2$, $r>0$, $t>0$, and $R>0$, set 
    \begin{align}
        \rho_{t,R}(r)
        &:=\max_{1\leq \ell \leq d-1} I_{\ell,t,R} (r) ,
    \end{align}
    where $I_{\ell,t,R}(r) $ denotes the integral
    \begin{align*}
        I_{\ell,t,R}(r):=& \int_{[\ball_{R}^{d}]^{\ell}}t^{\ell}\Bigg(t^{d-\ell}\int_{[\ball_{R}^d]^{d-\ell}}\ind{\norm{H_{1}\cap \ldots \cap H_d}\geq r} \mu_{d-1}^{\otimes (d-\ell)} (\dint (H_{\ell+1},\ldots, H_d))\Bigg)^2
        \\&\hspace{7cm}\times\mu_{d-1}^{\otimes \ell} (\dint( H_1,\ldots,H_{\ell})) .
    \end{align*}
    Then, for $R=t^{-\frac{d}{d+1}}$ and $t\geq 1$, we have that
    \begin{equation*}
        \rho_{t,R}(r) \leq C t^{-\frac{1}{d+1}} \ln(t) r^{-2} ,
    \end{equation*}
    where $C$ is a positive constant depending only on $d$.
\end{lemma}
\begin{proof} 
    Let $\ell \in \{1,\ldots,d-1\}$ and $r>0$ be fixed.
    We want to show that $I_{\ell,t,R}(r)$ tends to $0$, as $t\to\infty$ and $R=t^{-\frac{d}{d+1}}\to 0$.
    The constant $C$ appearing below depends on $d$ and $\ell$, and is independent of everything else. Its specific value varies from line to line and is irrelevant for the proof.

    \paragraph{First part of the proof: Make use of Lemma $\ref{density}$.} 
    In this first part we use integral geometric formulas in order to rewrite the integral $I_{\ell,t,R}(r)$ in term of integrals studied in Lemma \ref{density}.
    
    Let $E\in A(d,\ell)$ (resp. $F\in A(d,d-\ell)$) denotes the intersection of the affine hyperplanes involved in the inner (resp.\ outer) integral,
    that is 
    \[ E := H_{\ell+1}\cap\cdots\cap H_d\qquad\text{and}\qquad
     F := H_{1}\cap\cdots\cap H_\ell . \]
    Applying twice the Blaschke-Petkantschin formula \eqref{eq:B-P} permits us to rewrite our integral by \enquote{pivoting} the hyperplanes $H_i$'s around $E$ and $F$.
    It gives that $I_{\ell,t,R}(r)$ equals
    \begin{equation*}
    \begin{split}
        & t^{2d-\ell} \int_{A(d,d-\ell)} \int_{(A(F,d-1) \cap [\ball_{R}^{d}] )^{\ell}} \Bigg( \int_{A(d,\ell)} \int_{(A(E,d-1) \cap [\ball_{R}^d])^{d-\ell}} \ind{\norm{E\cap F}\geq r} 
        \\&\qquad\qquad\qquad\qquad\qquad\times [H_{\ell+1},\ldots,H_d]^{\ell+1} (\mu_{d-1}^E)^{\otimes (d-\ell)} (\dint (H_{\ell+1},\ldots, H_d)) \mu_{\ell} (\dint E) \Bigg)^2
        \\&\qquad\qquad\qquad\qquad\qquad\times [H_{1},\ldots,H_\ell]^{d-\ell+1} (\mu_{d-1}^F)^{\otimes \ell} (\dint( H_1,\ldots,H_{\ell})) \mu_{d-\ell}(\dint F)  .
    \end{split}
    \end{equation*}
    Note that we also took out of the integral the factors $t^\ell$ and $t^{d-\ell}$.
    We will now make more apparent the structure of this expression by using the following notation, which is similar to the one used in Lemma \ref{density}:
    \begin{equation*}
    \begin{split}
        J(E)
        &:=\int_{(A(E,d-1)\cap [\ball_{R}^d])^{d-\ell}}[H_{\ell+1},\ldots,H_{d}]^{\ell+1} 
        (\mu_{d-1}^{E})^{\otimes (d-\ell)} (\dint (H_{\ell+1}, \ldots, H_{d})), 
    \end{split}
    \end{equation*}
    and
    \begin{equation*}
    \begin{split}
        J(F)
        &:=\int_{(A(F,d-1)\cap [\ball_{R}^d])^{\ell}}[H_{1},\ldots,H_{\ell}]^{d-\ell+1}
        (\mu_{d-1}^{F})^{\otimes \ell} (\dint (H_{1}, \ldots, H_{\ell})) .
    \end{split}
    \end{equation*}
    With this notation and using Fubini's theorem the investigated integral takes the form
    \begin{equation*}
    \begin{split}
        & t^{2d-\ell} \int_{A(d,d-\ell)} J(F) \Bigg( \int_{A(d,\ell)} \ind{\norm{E\cap F}\geq r}  J(E) \mu_{\ell}(\dint E) \Bigg)^2 \mu_{d-\ell}(\dint F)  .
    \end{split}
    \end{equation*}
    Lemma \ref{density} gives us precise approximations of $J(F,R)$ and $J(E,R)$ in terms of the distances between the origin and the flats $F$ and $E$, respectively.
    In order to have a hand on these distances we decompose the affine flat $E\in A(d,\ell)$ (resp.\ $F\in A(d,d-\ell)$) into a sum a linear flat $E_o\in G(d,\ell)$ (resp.\ $F_o\in G(d,d-\ell)$) and a translation vector $x_E\in E_o^\perp$ (resp. $x_F\in F_o^\perp$) in the orthogonal complement of that flat.
    The aforementioned distances are simply the Euclidean norms of the vectors $x_E$ and $x_F$.
    The quantity of interest thus becomes
    \begin{equation*}
    \begin{split}
        & t^{2d-\ell} \int_{G(d,d-\ell)} \int_{F_o^\perp} J(F_o+x_F) \Bigg( \int_{G(d,\ell)} \int_{E_o^\perp} \ind{\norm{(E_o+x_E)\cap (F_o+x_F)}\geq r} 
        \\&\hspace{5cm}\times J(E_o+x_E) \dint x_E \, \nu_{\ell}(\dint E_o) \Bigg)^2 \dint x_F \, \nu_{d-\ell}(\dint F_o)  .
    \end{split}
    \end{equation*}
    It is easy to see that $J(E) = C \ind{\norm{x_E}\leq R}$ if $\ell=d-1$ and  $J(F) = C \ind{\norm{x_F}\leq R}$ if $\ell=1$.
    For the other cases we use  Lemma $\ref{density}$ which provides us for $\ell \geq 2$ with 
    %with $k=\dim E =\ell$, $a=\ell+1$ (and thus $d-k+a=d+1$) gives
    \begin{equation*}
        J(E) 
        = \ind{\norm{x_E} <R} C_{d-\ell,\ell+1}^{(1)} 
        + \ind{\norm{x_E}\geq R} \Bigl( \frac{R}{\norm{x_E}} \Bigr)^{d+1} \Bigl[ C_{d-\ell,\ell+1}^{(2)} + O \Bigl( \frac{R^2}{\norm{x_E}^2}\Bigr)  \Bigr] ,
    \end{equation*}
    and for $\ell\leq d-2$ with
    %with $k=\dim F =d-\ell$, $a=d-\ell+1$ (and thus $d-k+a=d+1$) gives
    \begin{equation*}
        J(F) 
        = \ind{\norm{x_F} <R} C_{\ell,d-\ell+1}^{(1)} 
        + \ind{\norm{x_F}\geq R} \Bigl( \frac{R}{\norm{x_F}} \Bigr)^{d+1} \Bigl[ C_{\ell,d-\ell+1}^{(2)} + O \Bigl( \frac{R^2}{\norm{x_F}^2}\Bigr)  \Bigr] .
    \end{equation*}
    Here, we only need to consider upper bounds, so we will keep from these two approximations the following statements which hold for all $\ell \in \{1,\ldots,d-1\}$:
    \begin{equation*}
        J(E) 
        \leq C \min \Bigl( 1 , \frac{R}{\norm{x_E}} \Bigr)^{d+1} 
        \qquad \text{and}\qquad
        J(F) 
        \leq C \min \Bigl( 1 , \frac{R}{\norm{x_F}} \Bigr)^{d+1} ,
    \end{equation*}
    where $C$ is some large enough constant depending only on $d$ and $\ell$.
    Plugging these bounds in the integral above yields
    \begin{equation} \label{eq:I-Bound}
        I_{\ell,t,R}(r)
        \leq C \, t^{2d-\ell} \int_{G(d,d-\ell)} \int_{F_o^\perp} \min \Bigl( 1 , \frac{R}{\norm{x_F}} \Bigr)^{d+1} K^2 \, \dint x_F \, \nu_{d-\ell}(\dint F_o)  ,
    \end{equation}
    where
    \begin{equation*}
    \begin{split}
     K &= K(d,\ell,F_o,x_F,r)
        \\ &:= \int_{G(d,\ell)} \int_{E_o^\perp}
        \ind{\norm{(E_o+x_E)\cap (F_o+x_F)}\geq r} \min \Bigl( 1 , \frac{R}{\norm{x_E}} \Bigr)^{d+1} \dint x_E \, \nu_{\ell}(\dint E_o)    .
    \end{split}
    \end{equation*}
    
    \paragraph{Second part of the proof: Bound $K$.}
    We will now focus on bounding the newly introduced term $K$. One difficulty in estimating this term comes from the fact that it is hard to get a good hand on the norm of the intersection of the affine flats $(E_o+x_E)$ and $(F_o+x_F)$.
    In order to overcome this problem, we will reduce it to intersections involving one linear flat and one affine flat.
    As we will see, the norm of such intersections are much more easily handled.
    Observe that the origin and the points $E_o\cap (F_o+x_F)$, $(E_o+x_E)\cap (F_o+x_F)$ and $(E_o+x_E)\cap (F_o+x_F)$ are the vertices of a parallelogram.
    Thus we have that
    \[ (E_o+x_E)\cap (F_o+x_F) 
    = (E_o \cap (F_o+x_F)) + ((E_o+x_E)\cap F_o) . \]
    Therefore, using the triangle inequality and the fact that if a sum of two positive real numbers is bigger than $r$ then at least one of the two is bigger than $r/2$, we obtain the following bound of the indicator function which appears in the term $K$:
    \begin{equation*}
    \begin{split}
        \ind{\norm{(E_o+x_E)\cap (F_o+x_F)}\geq r} 
        &\leq \ind{\norm{E_o \cap (F_o+x_F)}\geq \frac{r}{2}} 
        \\&\qquad + \ind{\norm{(E_o+x_E)\cap F_o} \geq \frac{r}{2}} .
    \end{split}
    \end{equation*}
    Thus we have that $K\leq K_1 + K_2$, where
    \begin{equation*}
    \begin{split}
         K_1 &= K_1(d,\ell,F_o,x_F,r)
        \\ & := \int_{G(d,\ell)} \int_{E_o^\perp}
        \ind{\norm{E_o \cap (F_o+x_F)}\geq \frac{r}{2}}
        \min \Bigl( 1 , \frac{R}{\norm{x_E}} \Bigr)^{d+1} \dint x_E \, \nu_{\ell}(\dint E_o) ,
    \end{split}
    \end{equation*}
    and
    \begin{equation*}
    \begin{split}
         K_2 &= K_2(d,\ell,F_o,r)
        \\& := \int_{G(d,\ell)} \int_{E_o^\perp}
        \ind{\norm{(E_o+x_E)\cap F_o} \geq \frac{r}{2}}
        \min \Bigl( 1 , \frac{R}{\norm{x_E}} \Bigr)^{d+1} \dint x_E \, \nu_{\ell}(\dint E_o) .
    \end{split}
    \end{equation*}
    Let us consider $K_1$. 
    By using spherical coordinates in the $(d-\ell)$-dimensional space $E_o$ and then Fubini's theorem, we get that $K_1 = K_{1,1}\, K_{1,2} $, with
    \begin{equation*} 
        K_{1,1} 
        = K_{1,1} (d,\ell,F_o,x_F,r)
        = \int_{G(d,\ell)} \ind{\norm{E_o \cap (F_o+x_F)}\geq \frac{r}{2}} \nu_{\ell}(\dint E_o)
    \end{equation*}
    and 
    \begin{equation*}
        K_{1,2}
        = K_{1,2} (d,\ell)
        %= \int_0^\infty \min \Bigl( 1 , \frac{R}{s} \Bigr)^{d+1} \omega_{d-\ell} s^{d-\ell-1} \dint s .
        = C \int_0^\infty \min \Bigl( 1 , \frac{R}{s} \Bigr)^{d+1} s^{d-\ell-1} \dint s .
    \end{equation*}
    The integral $K_{1,1}$ is the probability that the random linear space $E_o$ intersects the fixed flat $F_o+x_F$ at a distance from the origin greater than $r/2$.
    Such a probability is estimated in Lemma \ref{lem:intersection}, from which we get that
    \begin{equation*}
        K_{1,1} \leq C \min \Bigl( 1 , \frac{\norm{x_F}}{r} \Bigr) .
    \end{equation*}
    We evaluate also the integral $K_{1,2}$, namely
    \begin{equation*}
    \begin{split}
        K_{1,2}
        & =  \int_0^R s^{d-\ell-1} \omega_{d-\ell} \dint s 
        + \int_R^\infty R^{d+1} s^{-\ell-2} \omega_{d-\ell} \dint s
        %\\& =  \frac{\omega_{d-\ell}}{d-\ell} R^{d-\ell}
        %+ \frac{\omega_{d-\ell}}{\ell+1} R^{d+1} R^{-\ell-1}
        %=  \frac{\omega_{d-\ell} (d+1)}{(d-\ell)(\ell+1)} R^{d-\ell} .
        =  C R^{d-\ell} .
    \end{split}
    \end{equation*}
    Note that for the last equality to hold one needs that both exponents $d-\ell-1$ and $-\ell-2$ are distinct from $-1$. This is insured by the condition $1\leq \ell \leq d-1$.
    Putting the pieces together we get the bound
    \begin{equation} \label{eq:K1bound}
        K_1 \leq C \min \Bigl( 1 , \frac{\norm{x_F}}{r} \Bigr) \, R^{d-\ell} .
    \end{equation}
    
    Next, we deal with $K_2$ and
    recall that
    \begin{equation*}
    \begin{split}
         K_2 &= K_2(F_o)
         = \int_{G(d,\ell)} \int_{E_o^\perp}
        \ind{\norm{(E_o+x_E)\cap F_o} \geq \frac{r}{2}}
        \min \Bigl( 1 , \frac{R}{\norm{x_E}} \Bigr)^{d+1} \dint x_E \, \nu_{\ell}(\dint E_o) .
    \end{split}
    \end{equation*} 
    In this integral we consider the intersection of a fixed linear space $F_o\in G(d,d-\ell)$ with a \enquote{moving} affine space $E_o+x_E$.
    If we compare with the definition of $K_{1,1}$ we see some similarities but also that the roles of which space is fixed and which one is moving are interchanged.
    As above we will use Lemma \ref{lem:intersection} to bound this integral, but first we need to rewrite it in a more convenient form.
    Using spherical coordinates in the $(d-\ell)$-dimensional space $E_o^\perp$, Fubini's theorem and \enquote{parametrising} the space $\{ E_o + s u_E \mid E_o \in G(d,\ell) \, , \, u_E \in \sphere^{d-1} \cap E_o^{\perp} \}$ by $SO_d$ we can rewrite the integral as 
    \begin{equation*}
        K_2
        = C \int_0^\infty \int_{SO_d} 
        \ind{\norm{\theta E_s \cap F_o} \geq \frac{r}{2}}
        \nu (\dint \theta) \ \min \Bigl( 1 , \frac{R}{s} \Bigr)^{d+1} s^{d-\ell-1} \dint s ,
    \end{equation*} 
    where $E_s$ denotes an arbitrarily chosen $\ell$-dimensional flat at distance $s$ from the origin, $s>0$. 
    The inner integral is the probability that a random flat $\theta E_s$ and a fixed linear space $F_o$ have their intersection point with a norm at least $r/2$.
    Lemma \ref{lem:intersection} tells us that, for a fixed $r$, this is of order at most $\min(1,s/r)$.
    Therefore
    \begin{equation*}
        K_2
        \leq C \int_0^\infty \min \Bigl( 1 , \frac{R}{s} \Bigr)^{d+1} \min\Bigl(1,\frac{s}{r}\Bigr) s^{d-\ell-1} \dint s .
    \end{equation*}  
    Recall that $R=t^{-\frac{d}{d+1}}$ by assumption. Thus, without loss of generality, we may assume that $t$ is large enough to ensure that $R<r$, and therefore
    \begin{align*}
        K_2 
        %
        %& \leq C \left( \int_0^R 1^{d+1} s s^{d-\ell-1} \dint s 
        %+ \int_R^1 \Bigl(\frac{R}{s}\Bigr)^{d+1} s s^{d-\ell-1} \dint s 
        %+ \int_1^\infty \Bigl(\frac{R}{s}\Bigr)^{d+1} 1 s^{d-\ell-1} \dint s \right) 
        %
        \leq C \left( \int_0^R \frac{s^{d-\ell}}{r} \dint s 
        + \int_R^r \frac{R^{d+1} s^{-\ell-1}}{r} \dint s 
        + \int_r^\infty R^{d+1} s^{-\ell-2} \dint s \right) .
        %
        %\\&= C \left( \frac{R^{d-\ell+1}}{d-\ell+1}  
        %+ \frac{R^{d+1}}{\ell} \bigl( R^{-\ell} - 1 \bigr) 
        %+  \frac{R^{d+1}}{\ell+1} \right) .
    \end{align*}
    The two first integrals give a term of order $R^{d-\ell+1} r^{-1} $ and the third a term of order $R^{d+1} r^{-\ell-1}$.
    Since $R<r$ this yields
    \begin{equation} \label{eq:K2bound}
        K_2 \leq 
        % C \max \left( R^{d-\ell+1} , R^{d-\ell+1} , R^{d+1} \right) =
        C \frac{R^{d-\ell+1}}{r}.
    \end{equation}
    Combining the bounds \eqref{eq:K1bound} and \eqref{eq:K2bound} for the terms $K_1$ and $K_2$, we finally get
    \begin{equation*}
        K = K_1 + K_2 
        \leq C \left( \min \Bigl( 1 , \frac{\norm{x_F}}{r} \Bigr) R^{d-\ell} +  \frac{R^{d-\ell+1}}{r} \right) .
    \end{equation*}
    
    \paragraph{Third (and final) part of the proof.}
    We now plug the last bound into \eqref{eq:I-Bound} to conclude that
    \begin{equation} \label{eq:I-Bound-2}
        I_{\ell,t,R}(r)
        \leq C \, t^{2d-\ell} \int_{G(d,d-\ell)} L(F_o) \nu_{d-\ell}(\dint F_o)  ,
    \end{equation}
    where
    \begin{equation*}
        L(F_o)
        := \int_{F_o^\perp} \min \Bigl( 1 , \frac{R}{\norm{x_F}} \Bigr)^{d+1} \left( \min \Bigl( 1 , \frac{\norm{x_F}}{r} \Bigr) R^{d-\ell} +  \frac{R^{d-\ell+1}}{r} \right)^2 \, \dint x_F .
    \end{equation*}
    We use spherical coordinates in the $\ell$-dimensional subspace $F_o^\perp$ to see that
    \begin{align*}
        L(F_o)
        &= C \int_0^\infty \min \Bigl( 1 , \frac{R}{s} \Bigr)^{d+1} \left( \min\Bigl(1,\frac{s}{r}\Bigr) R^{d-\ell} +  \frac{R^{d-\ell+1}}{r} \right)^2 s^{\ell-1} \dint s.
    \end{align*}
    Recall that without loss of generality $R<r$. We split the integral into three parts:
    \begin{align*}
        L(F_o)
        &= C \biggl( \int_0^R 1^{d+1} \Bigl( \frac{s}{r} R^{d-\ell} + \frac{R^{d-\ell+1}}{r} \Bigr)^2 s^{\ell-1} \dint s
        \\&\qquad\qquad\qquad + \int_R^r \Bigl(\frac{R}{s}\Bigr)^{d+1} \Bigl( \frac{s}{r} R^{d-\ell} +  \frac{R^{d-\ell+1}}{r} \Bigr)^2 s^{\ell-1} \dint s 
        \\&\qquad\qquad\qquad + \int_r^\infty \Bigl( \frac{R}{s} \Bigr)^{d+1} \Bigl( R^{d-\ell} +  \frac{R^{d-\ell+1}}{r} \Bigr)^2 s^{\ell-1} \dint s \biggr).
    \end{align*}
    Since $R<\min(1,r)$ we use the inequalities
    \begin{align*}
        \frac{s}{r} R^{d-\ell} + \frac{R^{d-\ell+1}}{r} &\leq 2 \frac{R^{d-\ell+1}}{r} \qquad \text{ for } s\in [0,R] , \\
        \frac{s}{r} R^{d-\ell} + \frac{R^{d-\ell+1}}{r} &\leq 2 \frac{s R^{d-\ell}}{r} \qquad\;\; \text{ for } s\in [R,r] ,\\
        R^{d-\ell} + \frac{R^{d-\ell+1}}{r} &\leq 2 R^{d-\ell} \qquad\;\;\;\;\,\text{ for } s\in [r,\infty) ,
    \end{align*}
    to get the bound
    \begin{align*}
        &L(F_o)
        \leq C \biggl( \int_0^R \frac{R^{2d-2\ell+2}}{r^2} s^{\ell-1} \dint s
        + \int_R^r \frac{R^{3d-2\ell+1}}{r^2} s^{\ell-d} \dint s 
        + \int_r^\infty R^{3d-2\ell+1} s^{\ell-d-2} \dint s  \biggr).
    \end{align*}
    Note that each of the monomials $s^{\ell-1}$, $s^{\ell-d}$ and $s^{\ell-d-2}$ is distinct from $s^{-1}$, except for $s^{\ell-d}$ if $\ell=d-1$.
    In order to avoid distinguishing cases we will use the bound 
    \[ \int_R^r s^{\ell-d} \dint s \leq C R^{\ell-d+1} \ln(R^{-1}) \]
    which is valid for all $\ell\in\{1,\ldots,d-1\}$.
    Thus, we get
    \begin{align*}
        L(F_o)
        &\leq C \bigl( R^{2d-\ell+2} r^{-2}
        + R^{2d-\ell+2} \ln(R^{-1}) r^{-2}
        + R^{3d-2\ell+1} r^{\ell-d-1} \bigr)
        \\&=  C R^{2d-\ell+2} \bigl( r^{-2} + \ln(R^{-1}) r^{-2} + R^{d-\ell-1} r^{\ell-d-1} \bigr)
        \\&\leq  C R^{2d-\ell+2} \ln(R^{-1}) r^{-2}.
    \end{align*}
    We notice that this bound is independent from $F_o$, so when we plug it into \eqref{eq:I-Bound-2} we can ignore the integral over $F_o$.
    Also recall that $R=t^{-\frac{d}{d+1}}$.
    This gives
    \begin{align*}
        I_{\ell,t,R}(r)
        &\leq C \, t^{2d-\ell} R^{2d-\ell+2} \ln(R^{-1}) r^{-2}
        \\& = C \, t^{\frac{(2d-\ell)(d+1)-(2d-\ell+2)d}{d+1}} \ln(t) r^{-2}
        \\& = C \, t^{\frac{2d^2+2d-\ell d -\ell -2d^2+\ell d-2d}{d+1}} \ln(t) r^{-2}
        \\& = C \, t^{-\frac{\ell}{d+1}} \ln(t) r^{-2} .
    \end{align*}
    For $t\geq 1$ and $\ell\in\{1,\ldots,d-1\}$, the last quantity is maximised for $\ell=1$.
    This proves Lemma \ref{errorfunction}.
\end{proof}

%%%%%%%
\subsection{Convergence in distribution of the intersection process}
%%%%%%

Using Theorem \ref{convKR} we are now in a position to state our main theorem. % (which is a reformulation of Theorem \ref{thm:intro} from the introduction).
We prove that under the assumption that $R=t^{-\frac{d}{d+1}}$ the intersection point process $\Xi_{t,R}$ converges in  distribution to a Poisson point process  with power law density function  in $\RR^{d}\setminus \{0\}$. To explain the meaning of this convergence we supply the space $\sfN$ of counting measures on $\RR^{d}\setminus \{0\}$ with the vague topology induced by the mappings $\mu\mapsto\int f\dint\mu$, where $f:\RR^{d}\setminus \{0\}\to\RR$ is a non-negative continuous function with compact support. With this topology $\sfN$ becomes a Polish space and convergence in distribution of point processes refers to weak convergence induced by the vague topology, see \cite[Chapter 16 and Appendix A2]{kallenberg2002foundation}. Since all point processes we deal with are simple (that is, have no multiple points or atoms) the convergence $\xi_n\overset{d}{\to}\xi$ of a sequence of (simple) point processes $\xi_n$ to a simple point process $\xi$ is equivalent to the convergence in distribution of the real-valued random variables $\xi_n(B)$ to $\xi(B)$ for all relatively compact sets $B\subset\RR^d\setminus\{0\}$ satisfying $\xi(\partial B)=0$, where we write $\partial B$ for the boundary of $B$, see \cite[Theorem 16.16]{kallenberg2002foundation}.

\begin{theorem} \label{Poissonapproximation}
    Let $R=t^{-\frac{d}{d+1}}$ and $\zeta$ be a Poisson point process in $\RR^{d}\setminus \{0\}$ whose intensity measure has density function $x \to \frac{C_d}{\norm{x}^{d+1}}$, where $C_d$ is the same as in the previous section and given by \eqref{eq:DefConstCd}.  Then, we have that 
    \[\Xi_{t,R} \xrightarrow{d} \zeta, \quad t \to \infty \]
    in $\RR^{d}\setminus \{0\}.$
\end{theorem}
\begin{proof}%[Proof of Theorem \ref{Poissonapproximation}]
    Theorem 16.16 of \cite{kallenberg2002foundation} implies that in the following list of statements, (1) is equivalent to (2) and (3) is equivalent to (4):
    \begin{itemize}
        \item[(1)] $\Xi_{t,R}\mres{(\ball_{r}^d)^c}\xrightarrow{d}\zeta \mres{(\ball_{r}^d)^c}$ ;
        \item[(2)] $[\Xi_{t,R}\mres{(\ball_{r}^d)^c}] (B) \xrightarrow{d} [\zeta \mres{(\ball_{r}^d)^c}] (B) $ for all Borel sets $B\subset (\ball_r^d)^c$, relatively compact (in the space $(\ball_r^d)^c$) and  with boundary of zero Lebesgue measure.
        \item[(3)] $\Xi_{t,R} \xrightarrow{d} \zeta $ ;
        \item[(4)] $\Xi_{t,R} (B) \xrightarrow{d} \zeta (B) $ for all Borel sets $B\subset \RR^d \setminus \{0\} $, relatively compact (in the space $\RR^d \setminus \{0\} $) and with boundary of zero Lebesgue measure.
    \end{itemize}
    We are going to show that (1) holds and that (4) is equivalent to (2) (for all $r>0$).
    
    Let $r>0$ be any fixed real number. 
    Using Theorem \ref{convKR} we have that
    \[ \lim_{t \to \infty} d_{\mathrm{KR}}(\Xi_{t,R}\mres{(\ball_{r}^{d})^c} , \zeta\mres{(\ball_{r}^d})^c) = 0 . \]
    This implies (1) by Proposition 2.1 of \cite{decreusefond2016functional}.
    
    Now, consider a set $B$ as in (4). Since it is relatively compact in $\RR^d\setminus\{0\}$ its boundary does not contain the origin and therefore there exists a small radius $r>0$ such that $B\subset (\ball_r^d)^c$. From this we conclude that since (2) holds for all $r$, (4) holds as well. As mentioned above, (4) is equivalent to (3) and therefore the theorem follows.
\end{proof}

\begin{remark} 
    The constant $C_d$ appearing in the statement of Theorem \ref{Poissonapproximation} is  defined in \eqref{eq:DefConstCd}. In other words, $C_d$ is $(2 \omega_{d-1} / \omega_{d})^{d} / d! $ times the expected volume of a random parallelotope spanned by independent random vectors $Y_1,\ldots,Y_d$ which are uniformly distributed on $\sphere^{d-2}\times[-1,1]$ (i.e., according to the normalised $(d-1)$-dimensional Hausdorff measure on that set). In dimension $d=2$ the integral in the definition of $C_2$ can be computed explicitly:
    \[
    C_2= \frac{1}{2} \Bigl( \frac{2\cdot 2}{2\pi} \Bigr)^{2} \int_{-1}^1\int_{-1}^1|x-y|\,\frac{\dint x}{2}\frac{\dint y}{2} 
    %= \textcolor{blue}{\frac{4}{\pi^2}} \frac{2}{3}
    = \frac{4}{3\pi^2} \approx 0.135.
    \]
    In higher dimensions such a computation seems no longer tractable.
\end{remark}

\begin{remark} 
    The result of Theorem \ref{Poissonapproximation} remains true if one constructs the point process $\Xi_{t,R}$ as the intersections of $n=\lfloor t \rfloor$ independent and identically distributed random hyperplanes uniformly distributed in the set of hyperplanes intersecting the ball $\ball_R^d$, that is, if we replace the role of the Poisson hyperplane process $\eta_{t,R}$ by a binomial hyperplane process of same intensity measure.
    The proof would follow exactly the same lines and relies on the fact that the bound from \cite{decreusefond2016functional} on the Kantarovich-Rubinstein distance applies as well when the underlying process is a binomial point process. We restricted ourselves to the Poisson case for simplicity and because the Poisson hyperplane process and its induced tessellation are the more classical objects in stochastic geometry.
    %In order to keep the present manuscript relatively short, we omit this proof.
\end{remark}

\subsection{Convergence of the convex hull}

In the previous section we have established that the intersection point process $\Xi_{t,R}$ converges in distribution to the Poisson point process $\zeta$ with density function $x\mapsto C_d \norm{x}^{-(d+1)}$, as $t\to\infty$ and $R=t^{-\frac{d}{d+1}}$.
Now we will consider the weak convergence of the convex hull together with its $f$-vector. For a polytope $P\subset\RR^d$ and $k\in\{0,1,\ldots,d-1\}$ we write $f_k(P)$ for the number of $k$-dimensional faces of $P$. The vector $f(P):=(f_0(P),f_1(P),\ldots,f_{d-1}(P))$ is called the $f$-vector of $P$. Moreover, we recall the Hausdorff distance
$$
d_H(A,B):=\max\big\{\max_{x\in A}\min_{y\in B}\|x-y\|,\max_{y\in B}\min_{x\in A}\|x-y\|\big\},\qquad A,B\in\cC',
$$
on the space $\cC'$ of non-empty compact subsets of $\RR^d$. Endowing $\cC'$ with the topology generated by the Hausdorff distance and the induced Borel $\sigma$-field, a random non-empty compact set is a measurable mapping from an underlying probability space $(\Omega,\cF,\PP)$ to the measurable space $\cC'$, and convergence in distribution of a sequence of random compact sets $Z_n$ to a random compact set $Z$ refers to weak convergence with respect to this topology, see \cite[Chapter 16]{kallenberg2002foundation} and \cite[Chapter 12.3]{schneider2008stochastic}.

\begin{theorem} \label{thm:hull-Hausdorff}
    Let $R=t^{-\frac{d}{d+1}}$.
    It holds that ${\rm conv}\;\Xi_{t,R}$ converges, as $t\to\infty$, to ${\rm conv}\;\zeta$ in distribution. Moreover, $f_k(\mathrm{conv} \, \Xi_{t,R})$ converges  to $f_k(\mathrm{conv} \;\zeta)$ in distribution for all $k\in\{0,1,\ldots,d-1\}$, as $t\to\infty$.
\end{theorem}
\begin{proof}
    First, it follows from \cite[Lemma 3.1.4]{schneider2008stochastic} and \cite[Theorem 12.3.5]{schneider2008stochastic} that $\mathrm{conv} \, \Xi_{t,R}$ (for any $t>0$ and $R>0$) and $\mathrm{conv} \;\zeta$ are almost surely random non-empty compact sets in $\RR^d$.
    Now, let $R=t^{-\frac{d}{d+1}}$ and consider the asymptotics as $t\to \infty$.
    We have already seen that $\Xi_{t,R}\xrightarrow{d}\zeta$.
    Applying the Skorokhod representation theorem \cite[Therorem 4.30]{kallenberg2002foundation} we can find a probability space $(\Omega,\cA,\PP)$ and random elements $\widetilde{\Xi}_{t,R}$ and $\widetilde{\zeta}$ such that $\widetilde{\Xi}_{t,R}$ and ${\Xi}_{t,R}$ and $\widetilde{\zeta}$ and ${\zeta}$ have the same distribution (for any $t>0$), and such that $\widetilde{\Xi}_{t,R}\to\widetilde{\zeta}$ almost surely. 
    For almost all $\omega\in\Omega$ we have that
    \begin{enumerate}
        \item $\widetilde{\zeta}(\omega) \cap H_{+} \neq \emptyset $ for every open half-space $H_{+} \subset \RR^{d}$ such that $0 \in \partial H_{+}$,
        \item $\widetilde{\zeta}(\omega)$ is in general position, that is, no $k+2$ atoms of $\widetilde{\zeta}(\omega)$ belong to the same $k$-dimensional flat, $k\in\{1,\ldots,d-1\}$.
    \end{enumerate}
    These are precisely the assumptions of Lemma 4.1 in \cite{kabluchko2019cones} which implies that, as $t\to\infty$ and for almost all $\omega\in\Omega$,
    \begin{equation*}
        \mathrm{conv} (\widetilde{\Xi}_{t,R}(\omega)) \to \mathrm{conv} (\widetilde{\zeta}(\omega))
    \end{equation*}
    with respect to the Hausdorff distance $d_H$ and, for any $k\in\{0,1,\ldots,d-1\}$,
    \begin{equation*}
        f_k(\mathrm{conv}(\widetilde{\Xi}_{t,R}(\omega)) \to f_k(\mathrm{conv}( \widetilde{\zeta}(\omega)).
    \end{equation*}
    Going back to the original probability space we get the two convergences claimed in the theorem.
\end{proof}

We would like to rephrase the result of Theorem \ref{thm:hull-Hausdorff} in a different way. For this consider the tessellation induced by a stationary and isotropic Poisson hyperplane process in $\RR^d$ of intensity $t>0$. This is a random collection of pairwise non-overlapping random polytopes covering the whole space. With probability one precisely one of these polytopes contains the origin of $\RR^d$ in its interior. This random polytope is denoted by $Z_t$ and called the zero cell of the Poisson hyperplane tessellation of intensity $t$. Let us also recall that for a convex set $K\subset\RR^d$ we denote by $K^\circ:=\{x\in\RR^d\mid\langle x,y\rangle\leq 1\text{ for all }y\in K\}$ the convex dual of $K$. In particular, if $P$ is a polytope containing the origin in its interior,
\begin{equation}\label{eq:fkPolar}
f_k(P)=f_{d-k-1}(P^\circ)    
\end{equation}
for all $k\in\{0,1,\ldots,d-1\}$, see \cite[Corollary 2.13]{Ziegler}.

\begin{corollary}\label{cor:ZeroCell}
    Let $R=t^{-\frac{d}{d+1}}$ and put $\gamma_d:=\frac{1}{2}C_d\omega_d$, where $C_d$ is defined by \eqref{eq:DefConstCd}. Then ${\rm conv}\;\Xi_{t,R}$ converges, as $t\to\infty$, in distribution to $Z_{\gamma_d}^\circ$. Moreover, $f_k({\rm conv}\;\Xi_{t,R})$ converges in distribution to $f_{d-k+1}(Z_{\gamma_d})$ for all $k\in\{0,1,\ldots,d-1\}$, as $t\to\infty$.
\end{corollary}
\begin{proof}
    The first part is a direct consequence of Theorem \ref{thm:hull-Hausdorff} and the fact (illustrated by Figure \ref{fig:dual}) that the random polytopes $\mathrm{conv} \;\zeta$ and $Z_{\gamma_d}^\circ$ are identically distributed, see \cite[Theorem 1.23]{kabluchko2018beta}. The second claim also follows from Theorem \ref{thm:hull-Hausdorff} together with \eqref{eq:fkPolar}.
\end{proof}

\begin{figure} 
    \centering
    \includegraphics[width=\linewidth]{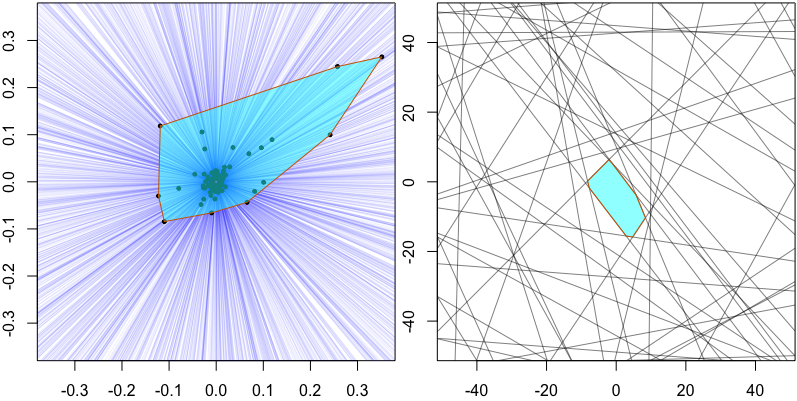}
    \caption{Left panel: One realisation of $\Xi_{t,R}$ with $t=125\,000\,000$ and $R=t^{-2/3}$ and its convex hull (cyan polygon). Right panel: Convex dual of the convex hull of $\Xi_{t,R}$ based on the same realisation (cyan polygon); the lines are dual to the intersection points in the left panel.}
    \label{fig:dual}
\end{figure}

\begin{remark} \label{remZeroTyp}
    By means of the last corollary, properties on the asymptotic distribution of $\mathrm{conv} \, \Xi_{t,R}$ can be derived by duality from statements on the distribution of the zero cell $Z_{\gamma_d}$.
    This random polytope, as well as its volume-debiased version -- known as typical cell~--, has been intensively investigated since the pioneering work of Miles and Matheron in the 1960'ies and 1970'ies, see \cite[Chapter 10.4]{schneider2008stochastic} and the many references cited therein. More recent results include formulas for the expected face numbers of the zero cell in any dimension \cite{kabluchko2019expected,kabluchko2018beta}, a description of
    `large' \cite{bonnet2018cells,HRTypicalPoisson,HRSZeroPoisson} or `small' \cite{bonnet2018small,schneider2019small} typical and zero cells,
    %concentration inequalities for the face numbers and intrinsic volumes of the zero cell \cite{GroteThaele}
    super-exponential bounds on the probability of having $n$ facets for large $n$ \cite{bonnet2018cells} and a probabilistic analysis of zero cell in high dimensions \cite{ElizaShinShell,HH,HHRT}.
\end{remark}

Based Theorem \ref{thm:hull-Hausdorff} we are now in the position to disprove the conjecture from \cite{devroye1993convex} we discussed in the introduction. In fact, Corollary \ref{cor:FaceNumbers} yields a lower bound on limit of all expected face number for any space dimension $d\geq 2$.

\begin{corollary}\label{cor:FaceNumbers}
    Let $R=t^{-\frac{d}{d+1}}$ and fix $k\in\{0,1,\ldots,d-1\}$. Then
    $$
    \liminf_{t\to\infty}\EE f_k(\mathrm{conv} \, \Xi_{t,R}) \geq \EE f_k(\mathrm{conv} \;\zeta).
    $$
    In particular, if $d=2$,
    $$
    \liminf_{t\to\infty}\EE f_0(\mathrm{conv} \, \Xi_{t,R}) \geq \EE f_0(\mathrm{conv} \;\zeta) = \frac{\pi^2}{2}.
    $$
\end{corollary}
\begin{proof}
    The first claim is a direct consequence of Fatou's lemma together with the fact that all the random variables $f_k(\mathrm{conv} \, \Xi_{t,R})$, $t>0$, and $f_k(\mathrm{conv} \;\zeta)$ are non-negative. In dimension $2$, we have that $f_0(P)=f_1(P)$ for any polygon $P\subset\RR^2$ and so the identity $\EE f_0(\mathrm{conv} \;\zeta) = \frac{\pi^2}{2}$ follows from \cite[Corollary 2.13]{kabluchko2019cones} by taking $\gamma=1$ there.
\end{proof}

We would like to emphasize that the expected face numbers $\EE f_k(\mathrm{conv} \;\zeta)$ of the convex hull of the Poisson point process $\zeta$ have been computed explicitly for any $d\geq 2$ and $k\in\{0,1,\ldots,d-1\}$ in \cite[Theorem 2.1]{kabluchko2019angles} and are given by
\[ \EE f_k(\mathrm{conv} \;\zeta) = \binom{d}{k+1} \pi^{d-1}\int_{-\infty}^{\infty}(\cosh u)^{-(d+1)}\Big(\frac{1}{2}+\frac{\mathfrak{i}u}{\pi}\Big)^{d-k-1}\,\dint u, \]
where $\mathfrak{i}$ is the imaginary unit. For example, for $d=2$ this leads to
$$
\EE f_0(\mathrm{conv} \;\zeta) = \EE f_1(\mathrm{conv} \;\zeta) = \frac{\pi^2}{2}
$$
as in Corollary \ref{cor:FaceNumbers} and for $d=3$ one has the values
$$
\EE f_0(\mathrm{conv} \;\zeta) = \frac{2}{3}(\pi^2+3),\qquad \EE f_1(\mathrm{conv} \;\zeta) = 2\pi^2,\qquad \EE f_2(\mathrm{conv} \;\zeta) = \frac{4}{3}\pi^2.
$$

\begin{remark}
    In analogy with \cite[Theorem 2.4]{kabluchko2019cones}, which deals with the convergence of the expected face numbers (and higher moments) of random cones generated by an i.i.d.\ sample, we conjecture that Corollary \ref{cor:FaceNumbers} can be strengthened to the statement that $\EE f_k^m(\mathrm{conv} \, \Xi_{t,R})\to \EE f_k^m(\mathrm{conv} \;\zeta)$, as $t\to\infty$, for any $k\in\{0,1,\ldots, d-1\}$ and $m\in\NN$. However, it should be noted that, although this is tempting, this does formally not follow from Theorem \ref{Poissonapproximation}. To prove the missing uniform integrability of the family of random variables $f_k^m(\mathrm{conv} \, \Xi_{t,R})$ seems a challenging task in view of the intricate correlation structure of the intersection point process $\Xi_{t,R}$. 
\end{remark}

\subsection*{Acknowledgement}
AB and GB were supported by the German Research Foundation (DFG) via GRK 2131 \textit{``High-Dimensional Phenomena in Probability -- Fluctuations and Discontinuity''}.

%%%%%%%%%%%%%%%%%%%
\printbibliography
%%%%%%%%%%%%%%%%%%%

\end{document}